\documentclass[letterpaper,11pt]{amsart}

\usepackage[all]{xy}                        %

\usepackage[margin=1in]{geometry}

\CompileMatrices                            

\UseTips                                    

\input xypic
\usepackage[bookmarks=true]{hyperref}       

\usepackage{amssymb,latexsym,amsmath,amscd}
\usepackage{xspace}
\usepackage{color}
\usepackage{tikz-cd}
\usepackage{graphicx}

\usepackage[utf8]{inputenc}
\usepackage{fourier}
\usepackage{array}
\usepackage{makecell}

\usepackage{array}
\newcolumntype{L}[1]{>{\raggedright\let\newline\\\arraybackslash\hspace{0pt}}m{#1}}
\newcolumntype{C}[1]{>{\centering\let\newline\\\arraybackslash\hspace{0pt}}m{#1}}
\newcolumntype{R}[1]{>{\raggedleft\let\newline\\\arraybackslash\hspace{0pt}}m{#1}}

\reversemarginpar

\theoremstyle{plain}
\newtheorem{theorem}{Theorem}[section]
\newtheorem*{theorem*}{Theorem}
\newtheorem{proposition}[theorem]{Proposition}
\newtheorem{corollary}[theorem]{Corollary}
\newtheorem{lemma}[theorem]{Lemma}

\theoremstyle{definition}
\newtheorem{definition}[theorem]{Definition}

\newtheorem{remark}[theorem]{Remark}
\newtheorem{example}[theorem]{Example}

\newcommand{\enm}[1]{\ensuremath{#1}}          %
\newcommand{\op}[1]{\operatorname{#1}}
\newcommand{\cal}[1]{\mathcal{#1}}

\newcommand{\CC}{\enm{\mathbb{C}}}

\newcommand{\QQ}{\enm{\mathbb{Q}}}
\newcommand{\ZZ}{\enm{\mathbb{Z}}}
\newcommand{\FF}{\enm{\mathbb{F}}}

\newcommand{\PP}{\enm{\mathbb{P}}}

\newcommand{\Dd}{\enm{\cal{D}}}
\newcommand{\Ee}{\enm{\cal{E}}}

\newcommand{\Hh}{\enm{\cal{H}}}
\newcommand{\Ii}{\enm{\cal{I}}}

\newcommand{\Ll}{\enm{\cal{L}}}

\newcommand{\Oo}{\enm{\cal{O}}}

\newcommand{\Zz}{\enm{\cal{Z}}}

\renewcommand{\phi}{\varphi}
\renewcommand{\theta}{\vartheta}
\renewcommand{\epsilon}{\varepsilon}

\newcommand{\End}{\op{End}}

\renewcommand{\to}[1][]{\xrightarrow{\ #1\ }}

\newcommand{\old}[1]{}

\begin{document}

\title[Hypersurface arrangements of aCM type]{Hypersurface arrangements of aCM type}

\author{E. Ballico and S. Huh}

\address{Universit\`a di Trento, 38123 Povo (TN), Italy}
\email{edoardo.ballico@unitn.it}

\address{Sungkyunkwan University, Suwon 440-746, Korea}
\email{sukmoonh@skku.edu}



\thanks{The first author is partially supported by GNSAGA of INDAM (Italy) and MIUR PRIN 2015 \lq Geometria delle variet\`a algebriche\rq. The second author is supported by the National Research Foundation of Korea(NRF) grant funded by the Korea government(MSIT) (No. 2018R1C1A6004285 and No. 2016R1A5A1008055). }


\begin{abstract}
We investigate the arrangement of hypersurfaces on a nonsingular varieties whose associated logarithmic vector bundle is arithmetically Cohen-Macaulay (for short, aCM), and prove that the projective space is the only smooth complete intersection with Picard rank one that admits an aCM logarithmic vector bundle. We also obtain a number of results on aCM logarithmic vector bundles over several specific varieties. As an opposite situation we investigate the Torelli-type problem that the logarithmic cohomology determines the arrangement.
\end{abstract}

\maketitle

\section{Introduction}
An arrangement $\Dd=\{D_1, \ldots, D_m\}$ of smooth hypersurfaces with normal crossings on a non-singular variety $X$, gives rise to the logarithmic sheaf $\Omega_X^1(\log \Dd)$ of differential $1$-forms with logarithmic poles along $\Dd$. This sheaf turns out to be locally free and was originally introduced by Deligne in \cite{De} to define a mixed Hodge structure on $X-\cup_{i=1}^mD_i$. In \cite{Te} Terao introduces the notion of freeness for an arrangement $\Dd$ of hyperplanes on a projective space $\PP^n$, not necessarily with normal crossings: the arrangement $\Dd$ is {\it free} if the dual of its associated logarithmic vector bundle is a direct sum of line bundles. The conjecture in \cite{OT} states that the freeness of $\Dd$ depends only on the combinatorics of $\Dd$, and it is widely open even in the case of $\PP^2$; refer to \cite{OT} for comprehensive understanding of this subject.                                                                        

In this paper we concentrate on a generalized notion of the freeness for arrangements of hypersurfaces over an arbitrary smooth projective variety. For a fixed polarization $\Oo_X(1)$ on a nonsingular variety $X$, a coherent sheaf $\Ee$ supporting on $X$ is called {\it arithmetically Cohen-Macaulay} (for short, aCM) if it has no intermediate cohomology, i.e. $H^i(\Ee(t))=0$ for all $t\in \ZZ$ and $i=1,\ldots, \dim X-1$. As its algebraic counterpart, over an aCM scheme $X\subset \PP^N$, it is well known that there exists a bijection between aCM sheaves on $X$ and maximal Cohen-Macaulay modules over its homogeneous coordinate ring. The famous Horrocks' theorem in \cite{Horrocks} asserts that the only aCM vector bundle on $\PP^n$ is a direct sum of line bundles, and this motivates to define another notion for arrangements. We say that an arrangement $\Dd$ is {\it of aCM type} if its logarithmic vector bundle $\Omega_X^1(\log \Dd)$ is aCM with respect to $\Oo_X(1)$. Note that an arrangement $\Dd$ of hyperplanes on $\PP^n$ with normal crossings is of aCM type if and only if it is free. 

Note from the Hodge theory and the existence of a polarization on $X$ that the empty arrangement is not of aCM type in any case. The main result of this paper is the following. 
\begin{theorem}\label{THM}
Let $X\subset \PP^{N}$ be a smooth complete intersection of dimension $n\ge 2$; in case $n=2$ assume further that $X$ is very general. If $\Dd$ is an arrangement of aCM type on $X$ with respect to $\Oo_X(1)$, then one of the following holds. 
\begin{itemize}
\item [(i)] $X=\PP^n$ and $\Dd=\{ H_1, \ldots, H_m\}$ is a hyperplane arrangement with $1 \le m \le n+1$; 
\item [(ii)] $X=Q_2$ a smooth quadric surface and $\Dd=\{ A_1, \ldots, A_a, B_1, \ldots, B_b\}$ is the set of $a+b$ distinct lines with $1\le a,b \le 3$ such that $A_i \in |\Oo_Q(1,0)|$ and $B_j\in |\Oo_Q(0,1)|$.  
\end{itemize}
\end{theorem}
\noindent Note that the assertion in Theorem \ref{THM} is not true in general due to counterexamples such as Fermat quartics in $\PP^3$; see Proposition \ref{kk1} and Remark \ref{cece}. As an automatic consequence, the only smooth complete intersection of dimension at least two with Picard rank one, which admits an arrangement of aCM type, is the projective space. For a general smooth variety with Picard rank one, we get non-existence of arrangements of aCM type with respect to an ample line bundle with enough global sections; see Proposition \ref{ee1}. While we also get non-existence results on surfaces of general type and abelian surfaces in Propositions \ref{p234} and \ref{p235}, there are plenty of examples of projective varieties with arrangements of aCM type, specially with higher Picard rank, e.g. the blow-up of $\PP^2$ at two points as in Proposition \ref{p236}. 

On the other hand, it is natural to consider the same vanishing condition for cohomology of logarithmic tangent bundle $TX(-\log \Dd)$, which is the dual of $\Omega_X^1(\log \Dd)$, in which case we call the arrangement of $T$-aCM type. From the definition, the notion of $T$-aCM is equivalent to the notion of aCM if the canonical sheaf is a multiple of the ample line bundle, i.e. $X$ is subcanonical. In case when $X$ is not subcanonical, one can expect new arrangements of $T$-aCM type, even the trivial one. In the end of Section $4$ we collect a number of results on arrangements of $T$-aCM type on Hirzebruch surfaces. 

In Section $5$ we investigate the graded module $H_*^i(\Dd)=\oplus_{t\in \ZZ}H^i(\Omega_X^1(\log \Dd)(t))$ associated to the logarithmic vector bundle of $\Dd$ for each $i=1,\ldots, \dim X-1$, called the {\it deficiency module} of degree $i$ associated to $\Dd$. The module $H_*^i(\Dd)$ is trivial for $\Dd$ of aCM type. One can also adapt the standard notion of $1$-Buchsbaum to $H_*^i(\Dd)$ as a weaker notion than aCM to produce a less simple deficiency module. In this section we obtain a number of Torelli-type results that the deficiency modules determine the arrangements on abelian varieties, K3 surfaces and Enriques surfaces.

\section{Preliminaries}
Throughout this article, we work over the field of complex numbers $\CC$. Let $X\subset \PP^N$ be a smooth projective variety of dimension $n\ge 2$ with a very ample line bundle $\Oo_X(1)$. For a coherent sheaf $\Ee$ on $X$ and $i\in \ZZ$, we set $H^i_* (\Ee):=\oplus_{t\in \ZZ} H^i(\Ee(t))$ with $\Ee(t):=\Ee \otimes \Oo_X(t)$.

\begin{definition}\label{deff}
A coherent sheaf $\Ee$ on a smooth projective variety $X\subset \PP^N$ is called {\it arithmetically Cohen-Macaulay} (for short, aCM) if we have $H^i_*(\Ee)=0$ for any $i=1,\ldots, \dim X-1$.
\end{definition}
\noindent Notice that being aCM does not depend on a twist of $\Ee$ by $\Oo_X(1)$.

\begin{definition}
A divisor $D$ on $X$ is said to have {\it normal crossings} if $\Oo_{D,x}$ is formally isomorphic to the quotient of $\Oo_{X,x}$ by an ideal generated by $t_1\cdots t_k$, where $t_1, \ldots, t_k$ is a subset of the set of local parameters in $\Oo_{X,x}$ for all $x\in D$. $D$ is also said to have {\it simple normal crossings} if it is the union of smooth divisors $D_i$, $i=1,\ldots, m$, which intersect transversally at each point. 
\end{definition}

\begin{definition}
An {\it arrangement} on $X$ is defined to be a set $\Dd = \{ D_1, \cdots, D_m\}$ of smooth irreducible divisors of $X$ with simple normal crossings such that $D_i \ne D_j$ for $i\ne j$. We can associate to $\Dd$ the {\it logarithmic sheaf} $\Omega_X^1(\log \Dd)$, the sheaf of differential $1$-forms with logarithmic poles along $\Dd$. The empty arrangement $\Dd=\emptyset$ is called the {\it trivial} arrangement and its associated logarithmic sheaf is simply $\Omega_X^1$. 
\end{definition}

If $\Dd$ has simple normal crossings, its logarithmic sheaf is known to be locally free and so it can be called to be the {\it logarithmic bundle}. It admits the residue exact sequence
\begin{equation}\label{seq1}
0\to \Omega_X^1 \to \Omega_X^1(\log \Dd ) \stackrel{\textrm{res}}{\to} \oplus_{i=1}^m {\epsilon_i}_* \Oo_{D_i} \to 0
\end{equation}
where $\epsilon_i : D_i \rightarrow X$ is the embedding and the map $\textrm{res}$ is the Poincar\'e residue morphism.

\begin{remark}\label{hod0}
In the Hodge theory for a smooth projective variety $X$ of dimension at least two with the Euclidean topology, the piece $H^{1,1}(X)$ of the Hodge decomposition can be identified with the sheaf cohomology $H^1(\Omega_X^1)$. Now an ample divisor on $X$ corresponds to a nonzero element in $H^{1,1}(X) \cap H^2(X, \QQ)$, and in particular we get $h^{1,1}(X):=h^1(\Omega_X^1)>0$. Thus $\Omega_X^1$ is not aCM. On the other hand, if $\Omega_X^1 (\log \Dd)$ is aCM, then by (\ref{seq1}) we see that $\Dd$ has at least $h^{1,1}(X)$ irreducible components. 
\end{remark}

\begin{remark}\label{tan}
Denoting the tangent bundle of $X$ by $TX$, the dual of a logarithmic bundle $\Omega_X^1 (\log \Dd)$ is the sheaf of logarithmic vector fields along $\Dd$, denoted by $TX(-\log \Dd)$; see \cite{D}. It admits the exact sequence
\begin{equation}\label{eqseq0}
0\to TX(-\log \Dd) \to TX \to \oplus_{i=1}^m {\epsilon_i}_*\Oo_{D_i}(D_i) \to 0.
\end{equation}
In case when $X$ is subcanonical, i.e. $\omega_X \cong \Oo_X(t)$ for some $t\in \ZZ$, by Serre's duality $TX(-\log \Dd )$ is aCM if and only if $\Omega_X^1(\log \Dd)$ is aCM. If this is the case, we have
\begin{equation}\label{eqseq1}
\sum _{i=1}^{m} h^0(\Oo _{D_i}(D_i)) \le h^0(TX).
\end{equation}
Note that the line bundle $\Oo _{D_i}(D_i)$ is the normal bundle of $D_i$ in $X$, and that the vector space $H^0(TX)$ is the tangent space at the identity of the functor $\mathrm{Aut}(X)$; see \cite[page 60]{brion}. For example, $\mathrm{Aut}(X)$ is countable if and only if $h^0(TX)=0$. In particular, we have $h^0(TX)=0$ if $X$ is of general type. Now assume $n=2$; in most cases there exists a non-trivial global vector field on $X$, while the vanishing condition $h^0(TX)=0$ would provide a strong restriction on the divisors $D_i$'s. In this article we obtain several partial results on the (non)existence of ($T$-)aCM arrangement of hypersurfaces over surfaces with $h^0(TX)=0$, which include the following:
\begin{enumerate}
\item [(i)] $X$ is of general type;
\item [(ii)] the minimal model of $X$ is a K3 surface or an Enriques surface;
\item [(iii)] most surface with $\kappa (X) = 1$ and $\kappa (X)=-\infty$;
\item [(iv)] $X$ is obtained by blowing up a Del Pezzo surface $X$ of dgree four at finitely many points.
\end{enumerate}
Note that the last class contains the smooth cubic surfaces in $\PP^3$ and the smooth complete intersection $X\subset \PP^4$ of two quadric hypersurfaces; see Proposition \ref{0001}.
\end{remark}

\begin{definition}
An arrangement $\Dd$ is said to be {\it of aCM type} with respect to $\Oo_X(1)$ if its associated logarithmic sheaf $\Omega_X^1(\log \Dd)$ is an aCM bundle on $X$ with respect to $\Oo_X(1)$. We also say that an arrangement $\Dd$ is {\it $T$-aCM} if the vector bundle $TX(-\log \Dd )$ is aCM. By definition and Serre's duality, over $X$ with $\omega _X \cong \Oo _X(e)$ for some $e\in \ZZ$, we get that $\Dd$ is aCM if and only if it is $T$-aCM.
\end{definition}

\begin{remark}\label{arar}
By Remark \ref{hod0} the trivial arrangement is never of aCM type. By Serre's vanishing theorem in \cite[III.5.2]{hart} and Serre's duality there is a positive integer $t_0$ such that for each $t\ge t_0$ we have $h^i(\Ee\otimes \Oo_X(mt))=0$ for all $i=1,\dots ,n-1$ and all $m\in \ZZ \setminus \{0\}$. This implies that there are many choice of very positive polarizations on $X$ for which a fixed arrangement $\Dd$ is ($T$-)aCM with respect to these polarizations.
\end{remark}

\begin{remark}
In several cases an arrangement $\Dd=\{D_1, \ldots, D_m\}$ can be shown to be not of aCM type, simply by showing that $h^i(\Omega_X^1 (\log \Dd))>0$ for some $i$. This motivates to define weaker notions: an arrangement $\Dd$ on $X$ is said to be {\it aCM in degree $0$} (resp. weakly aCM in degree $0$) if $h^i(\Omega _X^1(\log \Dd ))=0$ for all $1\le i \le n-1$ (resp. for $i=1$). Similarly we can define $T$-aCM in degree $0$ (resp. weakly $T$-aCM in degree $0$) by considering $TX(-\log \Dd)$ instead of $\Omega_X^1(\log \Dd)$. Note that these notions do not depend on the choice of a polarization of $X$. If $\Dd$ is weakly aCM in degree $0$, then we get $m\ge h^{1,1}(X)$ from (\ref{seq1}). If $\Dd$ is weakly $T$-aCM in degree $0$, then we get $h^0(TX) \ge \sum _{i=1}^{m} h^0(\Oo _{D_i}(D_i))$ from (\ref{eqseq0}).
\end{remark}

The logarithmic bundles of hyperplane arrangements on projective spaces have already been investigated by many authors and below we state some results of them. Conventionally, we will denote the hyperplane arrangement on $\PP^n$ by $\Hh$.
\begin{theorem}\cite{DK}\label{DKthm}
Let $\Hh=\{H_1, \cdots, H_m\}$ be a hyperplane arrangement on $\PP^n$. Then we have
$$\Omega_{\PP^n}^1(\log \Hh) \cong
\left\{
\begin{array}{ll}
\Oo_{\PP^n}^{\oplus (m-1)} \oplus \Oo_{\PP^n}(-1)^{\oplus (n-m+1)} &\hbox{ if $1\leq m \leq n+1$ } \\
T\PP^n(-1) &\hbox{ if $m=n+2$ }
 \end{array}
 \right.$$
\end{theorem}

\begin{example}\label{exaa}
Since $\Oo_{\PP^n}$ is the unique indecomposable aCM bundle on $\PP^n$ up to twist, the arrangement $\Dd$ is of aCM type if and only if it is free. By Theoren \ref{DKthm} any hyperplane arrangement $\Hh=\{H_1, \ldots, H_m\}$ on $\PP^n$ is of aCM type if $1\le m \le n+1$. If $\Hh$ is in general position with $m \ge n+2$, then it admits a Steiner resolution
\[
0\to \Oo_{\PP^n}(-1)^{\oplus (m-n-1)} \to \Oo_{\PP^n}^{\oplus (m-1)} \to \Omega_{\PP^n}^1(\log \Hh ) \to 0; 
\]
see \cite[Theorem 3.5]{DK}. In particular, we have $h^{n-1}(\Omega_{\PP^n}^1(\log \Hh)(-n))=m-n-1>0$ and so $\Hh$ is not of aCM type. On a smooth $n$-dimensional hyperquadric $Q_n \subset \PP^{n+1}$ with $n\ge 3$, no arrangements with simple normal crossings are of aCM type by \cite[Proposition 4.1]{BHM}. 
\end{example}

\begin{example}
For $X=\PP^2$, we have $H^1_*(\Omega_{\PP^2}^1)\cong \CC$ at degree $0$. Dually we have $H^1_*(T\PP^2)\cong \CC$ at degree $-3$. In particular, with respect to $\Oo_{\PP^2}(2)$, the trivial arrangement $\Dd =\emptyset$ is $T$-aCM, but not aCM.
\end{example}

\begin{remark}\label{r0}
Assume that $X$ is a smooth projective surface. By the Hodge theory and Serre's duality, we have $q(X)=h^1(\omega_X)=h^2(\Omega _X^1)$. Then for an arrangement $\Dd=\{  D_1, \ldots, D_m\}$ of aCM type, the exact sequence (\ref{seq1}) gives
\[
\chi (\Omega _X^1(\log \Dd )) = \chi (\Omega _X^1) + \sum _{i=1}^{m} \chi (\Oo _{D_i})
\]
and the inequality $\sum _{i=1}^{m} p_a(D_i) \le q(X)$. We also get that the classes $\{ [D_i]~|~ 1 \le i \le m\}$ generates $H^1(\Omega_X^1)$, which implies $m \ge h^{1,1}(X)$. Now assume moreover that $q(X)=0$, and then we have $p_a(D_i)=0$ for each $i$, i.e. each $D_i$ is a rational curve. On the other hand, by the Hodge theory we also have $h^0(\Omega_X^1)=0$ and so $\chi(\Omega_X^1)=-h^1(\Omega_X^1)=-h^{1,1}(X)$. In particular, we get
\[
h^0(\Omega_X^1(\log \Dd))+h^2(\Omega_X^1(\log \Dd))=m-h^{1,1}(X).
\]
\end{remark}

\begin{proposition}\label{ee1}
Assume that $\mathrm{Pic}(X) \cong \ZZ\langle \Oo _X(1)\rangle$ with $h^0(\Oo _X(1)) \ge h^0(TX)+2$. Then $X$ has no arrangement of ($T$-)aCM type.
\end{proposition}

\begin{proof}
Let $\Dd=\{D_1, \ldots, D_m\}$ be an arrangment on $X$. Since $\omega _X\cong \Oo _X(e)$ for some $e\in \ZZ$, the arrangement $\Dd$ is aCM if and only if it is $T$-aCM. Recall that the trivial arrangement $D=\emptyset$ is not of aCM type by Remark \ref{hod0}, and so we may assume $m\ge 1$. Take $D:=D_i$ for some $i$ and set $D\in |\Oo_X(a)|$ with $a>0$. From the assumption we get $h^0(\Oo _X(a))\ge 2+h^0(TX)$, and so the exact sequence 
\[
0\to \Oo _X\to \Oo _X(D)\to \Oo _D(D)\to 0
\]
gives $h^0(\Oo _D(D))\ge h^0(TX)+1$. Thus the exact sequence (\ref{eqseq0}) gives that $\Dd$ is not of aCM type.
\end{proof}

\begin{remark}
Assume that $\mathrm{Pic}(X) \cong \ZZ\langle \Oo _X(1)\rangle$ with $|\Oo _X(1)| \ne \emptyset$ and $h^0(TX)=0$. If $\Dd =\{D_1,\dots ,D_m\}$ is an arrangement which is $T$-aCM in degree $0$, then the same argument in the proof of Proposition \ref{ee1} shows that we have $m\in \{0,1\}$; in the former case we have $h^1(TX)=0$, and in the latter case we have $\Dd=\{D\}$ with $h^0(\Oo _X(D)) =1$.
\end{remark}


\section{Complete intersection}
To an arrangement $\tilde{\Dd}=\{\tilde{D}_1, \ldots, \tilde{D}_m\}$ on $\PP^N$ with no $\tilde{D}_i$ containing $X$, we may associate a new arrangement $\Dd=\{D_1, \ldots, D_m\}$ on $X$ with $D_i:=\tilde{D}_i \cap X$ and assume that $\Dd$ has simple normal crossings with $D_i\ne D_j$ for $i\ne j$, e.g. each $D_i$ intersects $X$ transversally. Then we have an exact sequence
\begin{equation}\label{norm}
0\to \Ii_{X,\PP^N}/\Ii_{X, \PP^N}^2 \to \Omega_{\PP^{n+1}}^1(\log \tilde{\Dd})_{|X} \to \Omega_X^1(\log \Dd) \to 0.
\end{equation}

\begin{lemma}\label{o1}
Let $X\subset \PP^{N}$ be a smooth complete intersection of dimension $n$ and $\Dd = \{D_1,\dots ,D_m\}$ be an arrangement of hypersurfaces on $X$ with each $D_i\in |\Oo _X(a_i)|$ for some positive integer $a_i$. Then there exists $\tilde{D}_i\in |\Oo_{\PP^{N}}(a_i)|$ with $D_i=\tilde{D}_i\cap X$ for each $i$ such that for any subset $J\subseteq \{1,\dots ,m\}$ we get either 
\begin{itemize}
\item [(i)] $(\cap _{i\in J}   \tilde{D}_i)\cap X =\emptyset$, in which case we have $\cap _{i\in J} D_i=\emptyset$, or 
\item [(ii)] each connected component of $(\cap _{i\in J} \tilde{D}_i)\cap X$ containing at least one point of $\cap _{i\in J} D_i$ has dimension $N-|J|$ and it is smooth at each point of $\cap _{i\in J} D_i$.
\end{itemize}
\end{lemma}

\begin{proof}
Note that the restriction map $H^0(\Oo _{\PP^{N}}(a_i))\rightarrow H^0(\Oo_X(a_i))$ is surjective, and thus there exists $\tilde{D}_i\in |\Oo_{\PP^N}(a_i)|$ with $D_i=\tilde{D}_i\cap X$ for each $i$. Now fix a subset $J\subseteq \{1,\dots ,m\}$, and then we get $(\cap _{i\in J} \tilde{D}_i)\cap X = \cap _{i\in J} D_i$ set-theoretically. 

Assume $\cap _{i\in J} D_i\ne \emptyset$ and take a point $q\in \cap_{i\in J} D_i$. Since $\Dd$ has simple normal crossings, we have $|J|\le n$ and $\cap _{i\in J} D_i$ is smooth of dimension $n-|J|$ at $q$. Note also that every irreducible component of $(\cap _{i\in J}  \tilde{D}_i)$ has dimension at least $N-|J|$. Since $X$ is a complete intersection, each irreducible component of $(\cap _{i\in J}  \tilde{D}_i)\cap X$ has dimension at least $n-|J|$. Since the reduction of $(\cap _{i\in J}  \tilde{D}_i)\cap X$ is smooth at $q$ with dimension $n-|J|$, the scheme $(\cap _{i\in J}  \tilde{D}_i)\cap X$ is locally a complete intersection at $q$. 

Thus it is sufficient to find suitable divisors $\tilde{D}_1, \ldots, \tilde{D}_m$ such that the scheme $(\cap _{i\in J}   \tilde{D}_i)\cap X$ contains $\cap _{i\in J} D_i$ with multiplicity one in a neighborhood of $q$; indeedn it is enough to check this at one point of each connected component of $\cap_{i\in J} D_i$. With no loss of generality assume $a_1\le \cdots \le a_m$, and choose an arbitrary finite subset $S\subset D_1\cup \cdots \cup D_m$ intersecting each connected component of $\cap _{i\in J} D_i$ for each $J$ with $|J|\le n$, i.e. $S$ contains at least one point from each connected component of $\cap_{i\in J} D_i$ for any $J$. Set $S_i:=S \cap D_i$. 

Now we proceed as follows: fix any $\tilde{D}_1$ satisfying $\tilde{D}_1 \cap X=D_1$, and then take a general $\tilde{D}_2$ with $\tilde{D}_2 \cap X=D_2$. Since $\tilde{D}_2$ is general in $|\Ii_{D_2,\PP^{N}}(a_2)|$ and $a_2\ge a_1$, $\tilde{D}_2$ is tranversal to $D_1\cap D_2$ at each point of $S_1\cap S_2$, concluding the proof if $m=2$. Now inductively we may choose a general $\tilde{D}_i \in |\Ii_{D_i, \PP^N}(a_i)|$ such that $\tilde{D}_i$ is transversal to each $(\cap_{j\in I} D_j)\cap D_i$ at each point of $(\cap_{j \in I} S_j) \cap S_i$ for any subset $I\subseteq \{1,\ldots, i-1\}$ with cardinality at most $n-1$. Then this choice of $\tilde{D}_1, \ldots, \tilde{D}_m$ satisfies the thesis of the lemma.
\end{proof}

\begin{proposition}\label{prop1}
No arrangement $\Dd$ on a smooth hypersurface $X_d\subset \PP^{n+1}$ with $d\ge 2$, associated to an arrangement $\tilde{\Dd}$ on $\PP^{n+1}$, is of aCM type on $X_d$.
\end{proposition}
\begin{proof}
Letting $X=X_d\subset \PP^{n+1}$, we have the following
\begin{equation}\label{norm1}
0\to \Oo_X(-d) \to \Omega_{\PP^{n+1}}^1(\log \tilde{\Dd})_{|X} \to \Omega_X^1(\log \Dd) \to 0.
\end{equation}
Since $\Oo_X$ is aCM with respect to $\Oo_X(1)$, from the long exact sequence of cohomology associated to (\ref{norm}) we get $H^i_*( \Omega_{\PP^{n+1}}^1(\log \tilde{\Dd})_{|X}) \cong H^i_*(\Omega_X^1(\log \Dd))$ for all $i=1,\ldots, n-2$, and an exact sequence
\begin{equation}\label{lexa}
0 \to H^{n-1}(\Omega_{\PP^{n+1}}^1(\log \tilde{\Dd})_{|X}) \to H^{n-1}(\Omega_X^1(\log \Dd)) \to H^n(\Oo_X(-d))\cong H^0(\Oo_X(2d-n-2))^\vee
\end{equation}
We also have an exact sequence
\begin{equation}\label{ses}
0\to \Omega_{\PP^{n+1}}^1 (\log \tilde{\Dd})(-d) \to \Omega_{\PP^{n+1}}^1 (\log \tilde{\Dd}) \to \Omega_{\PP^{n+1}}^1 (\log \tilde{\Dd})_{|X} \to 0. 
\end{equation}
Assume first that $\tilde{\Dd}$ is of aCM type on $\PP^{n+1}$. Then from (\ref{ses}) we get $H^i_* (\Omega_{\PP^{n+1}}^1 ( \log \tilde{\Dd})_{|X})=0$ for any $i=1,\ldots, n-1$. Now the twist of (\ref{lexa}) by $\Oo_X(t)$ becomes the following
\[
0\to H^{n-1}(\Omega_X^1(\log \tilde{\Dd})(t)) \to H^n(\Oo_X(t-d))\stackrel{\eta}{\to} H^n(\Omega_{\PP^{n+1}}^1(\log \tilde{\Dd})(t)_{|X}),
\]
where the map $\eta$ is the dual of the map $H^0(\Omega_{\PP^{n+1}}^1(\log \tilde{\Dd})^\vee(d-n-2-t)_{|X}) \rightarrow H^0(\Oo_X(2d-n-2-t))$. Choosing $t=2d-n-2$, we have a map $\eta ^\vee: H^0(\Omega_{\PP^{n+1}}^1(\log \tilde{\Dd})^\vee(-d)_{|X}) \rightarrow H^0(\Oo_X)$. If there exists a direct summand $\Oo_{\PP^{n+1}}(a)$ of $\Omega_{\PP^{n+1}}^1(\log \tilde{\Dd})$ with $a\le -d$, then there would be a nonzero map $\Omega_{\PP^{n+1}}^1 \rightarrow \Omega_{\PP^{n+1}}^1 (\log \tilde{\Dd}) \to \Oo_{\PP^{n+1}}(a)$ and so an injection $\Oo_{\PP^{n+1}}(-a) \rightarrow T\PP^{n+1}$, a contradiction due to the assumption $d\ge 2$. Thus each factor of $\Omega_{\PP^{n+1}}^1(\log \tilde{\Dd})$ has degree at least $1-d$ and the map $\eta^\vee$ cannot be surjective. In particular, we get $H^{n-1}_*(\Omega_X^1(\log \Dd))\ne 0$ and so $\Dd$ is not of aCM type on $X$. 

Now assume that $\tilde{\Dd}$ is not of aCM type. If $H^i_*(\Omega_{\PP^{n+1}}^1(\log \tilde{\Dd}))\ncong 0$ for some $1\le i\le n-1$, then set 
\[
t_0:=\min\{ t\in \ZZ~|~H^i(\Omega_{\PP^{n+1}}^1(\log \tilde{\Dd})(t)))\ne 0\}.
\]
Then from (\ref{ses}) we get $H^i(\Omega_{\PP^{n+1}}^1(\log \tilde{\Dd})(t_0)_{|X})\ne 0$, and this implies $H^i(\Omega_X^1(\log \Dd)(t_0))\ne 0$ by (\ref{norm}). Thus we may assume that $H^i_*(\Omega_{\PP^{n+1}}^1 (\log \tilde{\Dd}))=0$ for all $1\le i \le n-1$ and $H^n_*(\Omega_{\PP^{n+1}}^1(\log \tilde{\Dd}))\ne 0$. Set 
\[
t_1:=\max\{ t\in \ZZ~|~H^n(\Omega_{\PP^{n+1}}^1(\log \tilde{\Dd})(t)))\ne 0\}.
\]
Then by (\ref{ses}) we get $H^{n-1}(\Omega_{\PP^{n+1}}^1(\log \tilde{\Dd})(t_1+d)_{|X})\ne 0$ and so we get $H^{n-1}(\Omega_X^1(\log \Dd)(t_1+d))\ne 0$. Thus $\Dd$ is not of aCM type on $X$.  
\end{proof}

\begin{remark}\label{rem1}
Take two smooth projective varieties $X \subset Y$ in $\PP^N$ such that $X\in |\Oo_Y(d)|$ with $d\ge 2$ and $Y$ is subcanonical with $\Oo_Y$ aCM with respect to $\Oo_Y(1)$. This implies that $X$ is also subcanonical with $\Oo_X$ aCM with respect to $\Oo_X(1)$. Then by the same argument in the proof of Proposition \ref{prop1} we get that no arrangement $\Dd$ on $X$, associated to an arrangement $\tilde{\Dd}$ is of aCM type. For example, in the case when $\tilde{\Dd}$ is of aCM type on $Y$, the map $\eta^\vee : H^0(\Omega_Y^1(\log \tilde{\Dd})^\vee (-d)_{|X}) \rightarrow H^0(\Oo_X)$ is again not surjective, because we get $H^0(\Omega_Y^1(\log \tilde{\Dd})^\vee (-d)_{|X})=0$. Otherwise, we would have $H^0(\Omega_Y^1(\log \tilde{\Dd})^\vee (-d))\ne 0$ since $\Omega_Y^1(\log \tilde{\Dd})^\vee$ is also aCM with respect to $\Oo_Y(1)$ and so $H^1(\Omega_Y^1(\log \tilde{\Dd})^\vee (-2d))=0$. Then we get a non-trivial map $ \Omega_Y^1\rightarrow \Omega_Y^1(\log \tilde{\Dd}) \rightarrow \Oo_Y(-d)$, a contradiction. 
\end{remark}
 
\begin{corollary}\label{cor2}
No arrangement $\Dd$ on a smooth complete intersection $X\subset \PP^{n+1}$ of degree at least two, associated to an arrangement $\tilde{\Dd}$ on $\PP^{n+1}$, is of aCM type on $X$. 
\end{corollary}

\begin{proof}
By Proposition \ref{prop1} we may assume that $X\subset \PP^{n+1}$ is non-degenerate with codimension $s\ge 2$. Set $X=Y_1 \cap \ldots \cap Y_s$ with each $Y_i$ a hypersurface of degree $d_i$ and $d_1\ge \cdots \ge d_s\ge 2$ so that the homogeneous ideal $I_{X,\PP^{n+1}}$ is generated by $(n+1-s)$ forms of degree $d_1,...,d_s$. Since the ideal sheaf $\Ii_{X, \PP^{n+1}}(d_1)$ is globally generated, by Bertini's theorem a general $Y_1\in |\Ii _{X,\PP^{n+1}}(d_1)|$ is smooth outside $X$. Since $X$ is smooth, the hypersurface $Y_1$ is smooth at each point of $X$, because $X =Y_1\cap \ldots \cap Y_s$ scheme-theoretically. Thus we may assume that $Y_1$ is smooth. By the same argument we may choose each $Y_i$ so that $Y_1 \cap \ldots \cap Y_i$ is smooth. In particular, $X$ is a subvariety of a smooth variety $Y:=Y_1 \cap \ldots \cap Y_{s-1}$ with $X\in |\Oo_Y(d_s)|$. Since $d_s$ is at least two, we get the assertion by Remark \ref{rem1}. 
\end{proof}

\begin{remark}\label{ttt}
The assertion in Proposition \ref{prop1} does not hold in general if an arrangement $\Dd$ is not associated to an arrangement on $\PP^{n+1}$. For example, consider a smooth quadric surface $Q\subset \PP^3$ and take an arrangement $\Dd=\{ A_1, \ldots, A_a, B_1, \ldots, B_b\}$ of $a+b$ distinct lines with each $A_i\in |\Oo_Q(1,0)|$ and $B_j\in |\Oo_Q(0,1)|$. Then we have 
\[
\Omega_Q^1 (\log \Dd) \cong \Oo_Q(a-2, 0)\oplus \Oo_Q(0,b-2)
\]
by \cite[Proposition 6.2]{BHM}. Note that the case $(a,b)=(1,1)$ is not an arrangement associated to an arrangement of a hyperplane in $\PP^3$, because the hyperplane does not intersect $Q$ transversally and also each ruling is not given as a hyperplane section. In particular, $\Omega_Q^1(\log \Dd)$ is aCM if and only if $1\le a,b \le 3$. In fact, this is the only possibility for the arrangements of aCM type. If $\Dd=\{D_1, \ldots, D_m\}$ be an arrangement of aCM type on $Q$, then each $D_i$ is smooth and rational by Remark \ref{r0} and $q(Q)=0$. Note that $H^2(\Omega_Q^1(-1,-1)) \cong H^0(\Oo_Q(-1,1)\oplus \Oo_Q(1,-1))^\vee \cong 0$. Thus from the twist of (\ref{seq1}) we get $h^1(\Oo_Q(-1,-1)\otimes \Oo_{D_i})=0$ for each $i$. If $D_i$ is in $|\Oo_Q(a,b)|$, then we have $h^1(\Oo_Q(-1,-1)\otimes \Oo_{D_i}) = h^0(\Oo_{D_i}(a+b-2))$ by Serre's duality and so we get $a+b \le 1$. In particular, each $D_i$ is a line.
\end{remark}

\begin{remark}\label{o3}
Let $X\subset \PP^{N}$ be a smooth complete intersection defined by $s$ hypersurfaces of degree $d_1\ge \cdots \ge d_s\ge 2$. If $n\ge 3$, then by the Lefschetz theorem we get $\mathrm{Pic}(X) =\ZZ\langle\Oo_X(1)\rangle$; see \cite[Corollary 1.27]{voi}. In case $n=2$, assume that $d_1\ge 4$ for $k=1$ and $(d_1, d_2)\ne (2,2)$ for $k=2$. Furthermore assume that $X$ is very general, i.e. denoting by $\mathfrak{X}$ the variety parametrizing the complete intersection surfaces defined by hypersurfaces of degree $d_1, \cdots, d_s$, the surface $X$ is contained in the complement of countably many proper subvarieties of $\mathfrak{X}$. Then by Max Noether's theorem we have $\mathrm{Pic}(X) =\ZZ\langle \Oo_X(1)\rangle $; \cite[Theorem 3.32 and 3.33]{voi}.
\end{remark}

\begin{proposition}\label{0001}
Let $X\subset \PP^N$ with $N\in \{3,4\}$ be a Del Pezzo surface of degree $N$. Then there exists no arrangement of aCM type on $X$. 
\end{proposition}
\begin{proof}
First let $X\subset \PP^3$ be a smooth cubic surface with $\omega_X^\vee \cong \Oo _X(1)$ as the polarization. If an arrangement $\Dd = \{D_1,\dots ,D_m\}$ on $X$ is of aCM type, then we have $m \ge 7$ from $\mathrm{Pic}(X) \cong \ZZ ^{\oplus 7}$ and Remark \ref{r0}. Since we have $q(X)=0$, each $D_i$ is smooth and rational by Remark \ref{r0}. Thus we get ${D_i}^2 -\deg (D_i) =-2$. On the other hand, we have $H^0(TX)=0$; indeed, $X$ is a blow-up of $\PP^2$ at six general points so that any automorphism sending each line in $X$ to itself is the identity. In particular, we have $|\mathrm{Aut}(X)|<\infty$ and this implies that there is no nonzero global vector fields on $X$. Then by Remark \ref{tan} we also have $h^0(\Oo _{D_i}(D_i)) =0$, i.e. ${D_i}^2 < 0$. Thus each $D_i$ is a line. 

Since $\Dd$ is of aCM type and $h^1(\Oo_{D_i}(-1))=0$ for each $i$, we get $H^1(TX)=0$ from (\ref{eqseq0}). We also get $h^2(TX)=h^0(\Omega_X^1(-1))=0$ by Serre's duality and $h^0(\Omega_X^1)=h^{1,0}(X)=q(X)=0$. In particular, we get $\chi(TX)=0$. Now twist the sequence (\ref{eqseq0}) by $\Oo _X(1)$ to get the exact sequence
\begin{equation}\label{eqx}
0 \to TX(-\log \Dd )(1)\to TX(1) \to \oplus _{i=1}^{m} \Oo _{D_i}\to 0.
\end{equation}
Since $\Dd$ is of aCM type and $h^1(\Oo_{D_i})=0$ for each $i$ by (\ref{eqseq0}), we get $h^1(TX(1))=0$. By Serre's duality we also get $h^2(TX(1))=h^0(\Omega_X^1(-2))=0$, and thus we have $\chi (TX(1))=h^0(TX(1))$. Let $C\in |\Oo _X(1)|$ be a smooth plane section. Since $\det (TX(1)) \cong \Oo _X(3)$, the restriction $TX(1)_{|C}$ is a vector bundle of rank two on $C$ with degree $9$. Since $C$ is an elliptic curve, we get $\chi (TX(1)_{|C}) =9$ by Riemann-Roch. Then the exact sequence
\begin{equation}\label{ert}
0 \to TX \to TX(1) \to TX(1)_{|C}\to 0
\end{equation}
gives $\chi (TX(1)) =9$ and so $h^0(TX(1)) =9$. On the other hand, from the restriction of Euler's exact sequence we get $h^0(T\PP^3(1)_{|X})=h^0(\Oo_X(2)^{\oplus 4})-h^0(\Oo_X(1))=40-4=36$. Thus from the exact sequence
\[
0\to TX(1) \to T\PP^3(1)_{|X} \to \Oo_X(4) \to 0
\]
we get $h^0(TX(1))=36-31=5$, a contradiction. 

Now let $X\subset \PP^4$ be a smooth complete intersection of two quadric hypersurfaces, which can be also obtained by blow-up $\PP^2$ at five points such that no three of them are collinear; see \cite{d}. Then we have an exact sequence
\begin{equation}\label{del4}
0\to \Oo_{\PP^4}(-4) \to \Oo_{\PP^4}(-2)^{\oplus 2} \to \Ii_{X, \PP^4} \to 0.
\end{equation}
If $\Dd=\{ D_1, \ldots, D_m\}$ is an arrangement of aCM type with respect to $\Oo_X(1) \cong \omega_X^\vee$, then as in above we have $m\ge 6$, since we have $\mathrm{Pic}(X)\cong \ZZ^{\oplus 6}$. Similarly as in above, each $D_i$ is a line. We also get $\chi(TX)=0$ and $\chi(TX(1))=h^0(TX(1))$. For a smooth plane hyperplane section $C\in |\Oo _X(1)|$, the restriction $TX(1)_{|C}$ is a vector bundle of rank two on $C$ with degree $12$. Since $C$ is an elliptic curve, we get $\chi (TX(1)_{|C}) =12$ by Riemann-Roch. Thus from the exact sequence (\ref{ert}) we get $h^0(TX(1))=\chi (TX(1)) =12$. On the other hand, since $X\subset \PP^4$ is projectively normal, we get $h^0(\Oo_X(2))=13$ and $h^0(\Oo_X(1))=5$ from (\ref{del4}). Thus by Euler's exact sequence we get $h^0(T\PP^4(1)_{|X})=h^0(\Oo_X(2)^{\oplus 5})-h^0(\Oo_X(1))=60$. Then from the exact sequence
\[
0\to TX(1) \to T\PP^4(1)_{|X} \to \Oo_X(3)^{\oplus 2} \to 0
\]
we get $h^0(TX(1))=h^0(T\PP^4(1)_{|X}) - h^0(\Oo_X(3)^{\oplus 3})=60-50=10$, a contradiction. \end{proof}

Finally, by combining Corollary \ref{cor2}, Remark \ref{ttt}, Remark \ref{o3} and Proposition \ref{0001}, we obtain the assertions in Theorem \ref{THM}.  

\begin{proof}[Proof of Theorem \ref{THM}: ]
It remains to consider the case $X=\PP^n$. If $\Dd=\{D_1, \ldots, D_m\}$ is an aCM arrangement, then the twist of (\ref{seq1}) by $\Oo_{\PP^n}(1-n)$ would give
\[
0\to \oplus_{i=1}^m H^{n-1}(\Oo_{D_i}(1-n)) \to H^n(\Omega_{\PP^n}^1(1-n)) \cong H^0(T\PP^n(-2))^\vee \cong 0. 
\]
Since $H^{n-1}(\Oo_{D_i}(1-n))\cong H^0(\Oo_{D_i}(d_i-2))^\vee$ with $d_i=\deg (D_i)$ by Serre's duality, we get $d_i=1$ for each $i$. Then the assertion follows from Theorem \ref{DKthm} and Example \ref{exaa}. 
\end{proof}

On the other hand, there exists a smooth surface in $\PP^3$ for which the assertion in Theorem \ref{THM} does not hold; see Proposition \ref{kk1}. 

\begin{lemma}\label{k0}
For a smooth surface $X\subset \PP^3$ of degree $d$, 
\begin{itemize}
\item [(i)] we have $h^1(\Omega ^1_X(t)) =0$ for all $t >0$; 
\item [(ii)] in case $d=4$, we have $h^1(\Omega ^1_X(t)) =0$ for all $t <0$.
\end{itemize}
\end{lemma}

\begin{proof}
Recall that $H^2_*(\Omega_{\PP^3}^1)=0$ and $H^1_*(\Omega_{\PP^3}^1)\cong \CC$ at degree $0$. Thus the exact sequence
\[
0 \to \Omega _{\PP^3}^1(t-d) \to \Omega _{\PP^3}^1(t) \to \Omega ^1_{\PP^3}(t)_{|X}\to 0
\]
gives $H^1_*((\Omega^1_{\PP^3})_{|X}) =0$. Then we get part (i) from the conormal exact sequence 
\[
0\to \Oo _X(t-d)\to \Omega ^1_{\PP^3}(t)_{|X} \to \Omega ^1_X\to 0. 
\]
In case $d=4$, we have $TX \cong \Omega_X^1$ and so we get part (ii) from part (i) and Serre's duality. 
\end{proof}

\begin{proposition}\label{kk1}
Let $X \subset \PP^3$ be a smooth quartic surface with $\mathrm{Pic}(X) \cong \ZZ ^{\oplus m}\cong \ZZ\langle D_1, \ldots, D_m\rangle $ with $m \ge 2$, where each $D_i$ is a line. Then the arrangement $\Dd=\{ D_1, \ldots, D_m\}$ is aCM with respect to $\Oo_X(1)$. 
\end{proposition}

\begin{proof}
Note that we have $h^1(\Oo _{D_i}(t)) =0$ for all $t\ge -1$, because each $D_i$ is a line. Since the classes of $D_1,\dots,D_m$ generate $H^1(\Omega _X^1)$, we get $h^1(\Omega ^1_X(\log \Dd ))=0$ by (\ref{seq1}). In fact, we have $h^1(\Omega_X^1(\log \Dd)(t))=0$ for all $t\ge -1$ by applying Lemma \ref{k0} to (\ref{seq1}). By Serre's duality, we have
\[
H^1(TX(-\log \Dd)(1)) \cong H^1(\Omega_X^1 (\log \Dd)(-1))^\vee ~\text{ and }~ H^2(TX(-\log \Dd)) \cong H^0(\Omega_X^1 (\log \Dd))^\vee,
\]
where the former is trivial. The latter is also trivial, because $[D_1],\dots ,[D_m]$ are linearly independent in $H^1(\Omega ^1_X)$. Thus $TX(-\log \Dd)$ is $2$-regular and so we get $h^1(TX(-\log \Dd )(t)) =0$ for all $t\ge 2$ by Castelnuovo-Mumford's regularity lemma. Then by Serre's duality we have $h^1(\Omega ^1_X(\log \Dd )(t))=0$ for all $t \le -2$.
\end{proof}

\begin{remark}\label{cece}
In Proposition \ref{kk1} we may take as $X$ a Fermat quartic of $\PP^3$ by \cite[Theorem 1.1]{ssvl}; refer to \cite{s} for the weaker result, but still sufficient to get that the classes of the lines generate $H^1(\Omega ^1_X)$. Over $\ZZ$ it was done in \cite{mizu}, as quoted in \cite{ssvl}. Since $\omega _X\cong \Oo _X$ for a K3 surface $X$, the arrangement $\Dd$ in Proposition \ref{kk1} is also $T$-aCM.
\end{remark}


\section{Surfaces}
In this section we always assume that $X$ is a smooth projective surface. 

\begin{remark}\label{r000}
We recall here a few cases where we may choose the polarization $\Oo_X(1)$ with $\omega _X\cong \Oo _X(t)$ for some $t\in \ZZ$, e.g. either $\omega _X$ or $\omega _X^\vee$ is ample. Note that this does not occur if $\kappa (X)=1$. 

\quad {(a)} In case $\kappa (X)=2$, we get that $\omega _X$ is ample if and only if $X$ is a minimal model with no smooth rational curve $C\subset X$ with $C^2=-2$; refer to \cite[Proposition 1]{Bom}.

\quad ({b}) In case $\kappa (X) =0$, if $\omega _X\cong \Oo_X$, then $X$ is a minimal model. Conversely, a minimal model $X$ has $\omega _X\cong \Oo_X$ if and only if $X$ is either a K3 surface or an abelian surface; see \cite[page 590]{GH} and \cite[Theorem V.6.3]{hart}. Refer to \cite[page 585; case $q=1$]{GH} for hyperelliptic surfaces.

\quad ({c}) In case $\kappa (X)=-\infty$, $X$ must be rational with $\omega _X^\vee$ ample, i.e. $X$ is a smooth Del Pezzo surface; refer to \cite[Chapter 8]{Dol}.
\end{remark}

\begin{remark}\label{x1}
Let $\Dd=\{ D_1, \ldots, D_m\}$ be an $T$-aCM arrangement on a surface $X$ with $g_i:= p_a(D_i)$. Assume that $\omega _X\cdot D_i < 0$ for all $i$; this is the case for all Del Pezzo surfaces and refer to \cite{d} for wide review on the Del Pezzo surfaces. Then by the adjunction formula, we have $2g_i-2 = {D_i}^2 +\omega _X\cdot D_i<{D_i}^2$ for each $i$, and this implies by Riemann-Roch that 
\[
h^0(\Oo _{D_i}(D_i)) = {D_i}^2+1-g_i =\frac{ {D_i}^2 -\omega _X\cdot D_i}{2}.
\]
In particular, if $g_i=1$, we have $h^0(\Oo _{D_i}(D_i)) ={D_i}^2\ge 1$. If $g_i\ge 2$, we get $h^0(\Oo _{D_i}(D_i)) \ge g_i$. Assume now that $X$ is a Del Pezzo surface with $h^0(TX)=0$, i.e. $X$ is the blow-up of $\PP^2$ at at least four points. Then by (\ref{eqseq0}) we get $h^0(\Oo_{D_i}(D_i))=0$ and so $g_i=0$ for each $i$. Indeed, we get $D_i^2=-\omega_X\cdot D_i=-1$ and hence each $D_i$ is embedded as a line. 
\end{remark}

\begin{proposition}\label{p234}
Let $X$ be a smooth surface of general type whose minimal model contains no curve with geometric genus at most $q(X)$. Then no arrangement on $X$ is of aCM type.
\end{proposition}

\begin{proof}
Let $\Dd=\{D_1, \ldots, D_m\}$ be an arrangement of aCM type on $X$. By Remark \ref{r0} we have $\sum _{i=1}^{m} p_a(D_i) \le q(X)$. Let $\pi : X\rightarrow \tilde{X}$ be the map to the minimal model $\tilde{X}$. By assumption the only curves of $X$ with geometric genus at most $q(X)$ are the smooth rational curves contracted by $\pi$. In particular, each $D_i$ is rational and contracted by $\pi$. Define
\[
\Gamma :=\langle [D_1], \cdots, [D_m]\rangle \subseteq \mathrm{Num}(X)\otimes \CC. 
\]
Then the restriction of the intersection form of $X$ to $\Gamma$ is negative definite and so $\Gamma$ must be a proper subset of $\mathrm{Num}(X)\otimes \CC$, contradicting Remark \ref{r0}.
\end{proof}

\begin{proposition}\label{p235}
Let $X$ be a smooth surface whose minimal model is an abelian surface. Then no arrangement on $X$ is of aCM type.
\end{proposition}

\begin{proof}
Assume that the map $\pi: X\rightarrow \tilde{X}$ to the minimal model is a sequence of $k$ contraction of exceptional curves. Then we have $h^1(\Omega ^1_X) = k+h^1(\Omega ^1_{\tilde{X}})$. Since $\tilde{X}$ is an abelian surface, $\pi$ is the Albanese mapping of $X$ with $q(X)=q(\tilde{X})=2$. Note that $\Omega _{\tilde{X}}^1 \cong \Oo _{\tilde{X}}^{\oplus 2}$, and so we get $h^1(\Omega _{\tilde{X}}^1) =4$. In particular, we have $h^1(\Omega ^1_X) =4+k$. If $\Dd = \{D_1,\dots ,D_m\}$ is an arrangement of aCM type on $X$, then we have $\sum _{i=1}^{m} p_a(D_i) \le 2$ by Remark \ref{r0}. On the other hand, since $\tilde{X}$ is an abelian variety, we get that $\pi (D_i)$ is a single point for $D_i$ rational. Since the classes of $D_1,\dots ,D_m$ generate $H^1(\Omega ^1_X)$, the classes of the images of the curves $D_i$ not contracted by $\pi$ generate $H^1(\Omega _{\tilde{X}}^1)$. This implies that $\sum _{i=1}^{m} p_a(D_i) \ge 4$, a contradiction.
\end{proof}

Let $X=\mathrm{Bl}_2\PP^2$ be the blow-up of $\PP^2$ at two distinct points, say $p_1$ and $p_2$. It has three exceptional curves $D_1$, $D_2$ and $D_3$; $D_i$ is the exceptional divisor over the point $p_i$ for each $i=1,2$, and $D_3$ is the strict transform of the line $L$ containing $\{p_1,p_2\}$. We have
\[
c_1^2(X) =7~,~c_2(X) =c_2(TX) = c_2(\PP^2) +2 =5.
\]
From Riemann-Roch and $\chi (\Oo _X)=1$, we get
\[
\chi (TX) = \frac{c_1\cdot (c_1-\omega _X)}{2} + 2\chi (\Oo _X) -c_2(X) = 7+2 -5=4.
\]
Note that we have $h^0(TX) =4$, because $\dim \mathrm{Aut}(X)=4$ and $H^0(TX)$ is the tangent space at the identity map $[\mathrm{id}_X]\in \mathrm{Aut}(X)$. By Serre's duality we also get $h^2(TX)=h^0(\Omega ^1_X\otimes \omega _X)\le h^0(\Omega ^1_X)=0$. Then these imply the vanishing $h^1(TX)=0$. Now choose $\Oo _X(1):= \omega _X^\vee$ as the polarization and take $\Dd =\{D_1,D_2,D_3\}$.

\begin{proposition}\label{p236}
The arrangement $\Dd$ of the three exceptional divisors on $X=\mathrm{Bl}_2\PP^2$ is of aCM type with respect to $\Oo_X(1)$. 
\end{proposition}

\begin{proof}
Since $\Omega_X^1$ is a vector bundle of rank two with $\Omega_X^1 \cong TX(-1)$, we have $h^1(\Omega_X^1(1))=h^1(TX)=0$. We also have $h^2(\Omega_X^1)=h^1(\omega_X)=h^1(\Oo_X)=0$ by Hodge's theorem and Serre's duality. Thus the bundle $\Omega_X^1$ is $2$-regular, and by the Castelnuovo-Mumford regularity lemma we have $h^1(\Omega_X^1(t))=0$ for all $t>0$. Note that from the cotangent exact sequence 
\[
0\to \pi^*\Omega_{\PP^2}^1 \to \Omega_X^1 \to \oplus_{i=1}^2 \epsilon_{i*}\Omega_{D_i}^1\to 0
\]
where $\epsilon_i: D_ i \rightarrow X$ is the embedding, we get $h^1(\Omega_X^1)=3$. Since the classes $\{[D_1],[D_2],[D_3]\}$ freely generate $H^1(\Omega ^1_X)$, we get $h^2(TX(-\log \Dd)(-1))=h^0(\Omega_X^1(\log \Dd))=0$ from $h^0(\Omega ^1_X)=0$ and the exact sequence (\ref{seq1}). By applying $h^1(TX) = 0$ and $\deg \Oo _{D_i}(D_i) =-1$ for each $i$ to the sequence (\ref{eqseq0}), we also get $h^1(TX(-\log \Dd ))=0$. Thus the bundle $TX(-\log \Dd)$ is $1$-regular, and in particular we have $h^1(TX(-\log \Dd )(t)) =0$ for all $t\ge 0$ by the Castelnuovo-Mumford regularity lemma. Now assume that $t<0$ and set $t'=-t\ge 1$; by Serre's duality we have 
\[
h^1(TX(-\log \Dd )(-t')) = h^1(\Omega _X^1(\log \Dd )(t'-1)).
\]
From the vanishing $h^1(\Oo_{D_i}(t'-1))=0$ for each $i$, we get an exact sequence
\[
\oplus_{i=1}^3 H^0(\Oo_{D_i}(t'-1)) \stackrel{\delta}{\to} H^1(\Omega_X^1(t'-1)) \to H^1(\Omega_X^1(\log \Dd)(t'-1)) \to 0.
\]
For $t'\ge 2$ we have $h^1(\Omega_X^1(t'-1))=0$ and so $H^1(\Omega_X^1(\log \Dd)(t'-1))=0$. In case $t'=1$ we may use that the classes $\{[D_1], [D_2], [D_3]\}$ generate $H^1(\Omega ^1_X)$ so that the coboundary map $\delta$ is an isomorphism.
\end{proof}

Although the trivial arrangement is never of aCM type by Remark \ref{hod0}, it is still possible for the trivial arrangement is of $T$-aCM type. Below we study the arrangement of ($T$)-aCM type on Hirzebruch surfaces. 

Let $\FF_e$ with $e\ge 0$ be the Hirzebruch surface with minimal self-intersection section $h$ of its ruling
$\pi : \FF_e \rightarrow \PP^1$ with $h^2=-e$. The surface $\FF_1$ is isomorphic to the blowing up of $\PP^2$ at one point and so $h^0(T\FF_1)=6$. We have $\FF_0\cong \PP^1\times \PP^1$ with $h^0(T\FF_0) =6$. Note that $\FF_0$ and $\FF_1$ are the del Pezzo surfaces of degree $8$. We can also interpret the Hirzebruch surfaces as $\FF_e =\PP(\Oo _{\PP^1}\oplus \Oo _{\PP^1}(-e)) $ and this implies $h^1(\Oo_{\FF_e})=g(\PP^1)=0$. For $e>0$ we have 
\[
h^0(T\FF_e) = \dim \mbox{Aut}(\FF_e) =\dim \mbox{Aut}(\PP^1) +\dim \End (\Oo _{\PP^1}\oplus \Oo _{\PP^1}(-e))-1 =e+5; 
\]
recall that for every smooth projective variety $X$, the set of the global vector fields $H^0(TX)$ is the tangent space at the identity of the functor of all automorphisms of $X$, and hence we have $h^0(TX)=\dim \mathrm{Aut}(X)$; see \cite[page 60]{brion}

We also have $\mathrm{Pic}(\FF_e)\cong \ZZ^{\oplus 2}\cong \ZZ \langle h,f\rangle$, where $f$ is a fiber of a ruling of $\FF_e$ for which $h$ is a section; we have $h^2=-e$, $h\cdot f=1$ and $f^2=0$. Note that a line bundle $\Oo _{\FF_e}(ah+bf)$ is ample if and only if it is very ample if and only if $a>0$ and $b>ae$. Since $\omega _{\FF_e}\cong \Oo _{\FF_e}(-2h-(e+2)f)$,there is a subcanonical polarization of $\FF_e$ if and only if $e=0,1$. 

Since the ruling $\pi : \FF_e\rightarrow \PP^1$ is a submersion, it induces a surjective map $\pi _{\ast}: T\FF_e \rightarrow  \pi^\ast (T\PP^1) \cong \Oo _{\FF_e}(2f)$. So from $\omega _{\FF_e}^\vee \cong \Oo _{\FF_e}(2h+(e+2)f)$ we get that $T\FF_e$ fits in an exact sequence
\begin{equation}\label{eqy1}
0 \to \Oo _{\FF_e}(2h+ef) \to T\FF_e \to \Oo _{\FF_e}(2f)\to 0.
\end{equation}

\begin{remark}
From (\ref{eqy1}) and its dual, we may compute all twisted cohomology groups of $T\FF_e$ and $\Omega_{\FF_e}^1$. For example, we get $h^0(T\FF_e\otimes \Ll^\vee )=0$ for every ample line bundle $\Ll$ on $\FF_e$. Thus if an arrangement $\Dd=\{D_1, \ldots, D_m\}$ is $T$-aCM with respect to an ample line bundle $\Ll$, we get $h^0(\Oo _{D_i}(D_i)\otimes \Ll ^\vee )=0$ for all $i=1,\dots m$ from the sequence (\ref{eqseq1}).
\end{remark}

\begin{remark}\label{ii0}
Assume that $e>0$. For $a\ge 0$, we have $h^1(\Oo _{\FF_e}(ah+bf)) =\sum _{i=0}^{a} h^1(\Oo _{\PP^1}(b-ie))$. In particular, if $a\ge 0$ and $b\ge ae-1$, then we have $h^1(\Oo _{\FF_e}(ah+bf)) =0$. We also have $h^i(\Oo _{\FF_e}(-h+bf)) =0$ for any $b\in \ZZ$ and $i\in \{0,1,2\}$ by the Leray spectral sequence of $\pi$, because $R^i\pi _{\ast}(\Oo _{\FF_e}(-h+bf)) =0$ for any $i\in \{0,1,2\}$. If $a\le -2$ we may get its computation using Serre's duality. In case $a=-1$ the cohomology always vanishes. 
\end{remark}

\begin{lemma}\label{ii}
The trivial arrangement $\Dd=\emptyset$ on $\FF_e$ is $T$-aCM in degree $0$ if and only if $e\in \{0,1\}$. 
\end{lemma}
\begin{proof}
Note that the trivial arrangement $\Dd =\emptyset$ is $T$-aCM in degree $0$ if and only if $h^1(T\FF_e) =0$. Assume that $e>0$. Then we have
\[
h^0(\Oo_{\FF_e}(2h+ef))=h^0(\Oo_{\PP^1}(e))+h^0(\Oo_{\PP^1})=e+2
\]
and $h^0(\Oo_{\FF_e}(2f))=h^0(\Oo_{\PP^1}(2))=3$. Note also that $h^1(\Oo_{\FF_e}(2f))=0$. In particular, from (\ref{eqy1}) we get an isomorphism $H^1(\Oo_{\FF_e}(2h+ef)) \cong H^1(T\FF_e)$. If $e=1$, then we get $h^1(\Oo_{\FF_1}(2h+f))=h^1(\Oo_{\FF_1}(2f))=0$, and so $h^1(T\FF_1)=0$ from (\ref{eqy1}). If $e\ge 2$, then we have $h^1(\Oo _{\FF_e}(2h +ef)) = h^1(\Oo _{\PP^1}(-e))  =e-1\ge 1$. 
This implies that $h^1(T\FF_e) \ne 0$. 
\end{proof}

\begin{remark}\label{rat0}
If $D$ is a smooth rational curve on $\FF_e$, say $D\in |\Oo_{\FF_e}(ah+bf)|$, then we get $-2=D^2+\omega_X\cdot D$, and so $(a-1)(ae-2b+2)=0$. Thus we get one of the following:
\begin{itemize}
\item [(i)] $D\in |\Oo_{\FF_e}(f)|$; $h^0(\Oo_D(D))=1$, 
\item [(ii)] $D\in |\Oo_{\FF_e}(h)|$; $D=h$ and $h^0(\Oo_D(D))=0$ if $e>0$, 
\item [(iii)] $D\in |\Oo_{\FF_e}(h+bf)|$ with $b\ge e$; 
\item [(iv)] $e=0$ and $D\in |\Oo_{\FF_0}(ah+f)|$ with $a\ge 1$, 
\item [(v)] $e=1$ and $D\in |\Oo_{\FF_1}(2h+2f)|$. 
\end{itemize}
In case (i) we have $\deg (\Oo_D(D))=0$ and $h^0(\Oo_D(D))=1$. In case (ii), we have $\deg (\Oo_D(D))=-e$ and so $h^0(\Oo_D(D))=0$. In case (iii) we have $\deg (\Oo_D(D))=2b-e$ and so $h^0(\Oo_D(D))=2b-e+1$. 
\end{remark} 

\begin{lemma}\label{popo}
Let $\Dd=\{D_1, \ldots, D_m\}$ be a $T$-aCM arrangement on $\FF_e$ with respect to some polarization. Then each $D_i$ is rational.
\end{lemma}

\begin{proof}
From the sequence (\ref{eqseq1}) we have $h^0(\Oo_{D_i}(D_i))\le h^0(T\FF_e)=e+5$ for each $i$. Set $D=D_i$ for some $i$, and assume that $D\in |\Oo_{\FF_e}(ph+qf)|$ with $p\ge 2$ and $q\ge pe$. Since $h^1(\Oo _{\FF_e}) =0$, the exact sequence 
\[
0 \to \Oo _{\FF_e}\to \Oo _{\FF_e}(D)\to \Oo _D(D)\to0
\]
gives $h^0(\Oo _D(D))= h^0(\Oo _{\FF_e}(D))-1$. First assume $e=0$, i.e. $\FF_0\cong \PP^1 \times \PP^1$. Then $D$ is not rational if and only if $p,q\ge 2$. Since $h^0(\Oo _{\FF_0}(ph+qf))=(p+1)(q+1)$ and $h^0(T\FF_0)=6$, the curve $D$ must be rational; if $D$ is rational with $p=1$, then we have $h^0(\Oo _D(D))= 2q+1$ and so the only possibility would be $q\in \{0,1,2\}$. Now assume $e>0$. For all integers $p\ge 0$ and $q\in \ZZ$, we have
\begin{equation}\label{eqy2}
\pi _{\ast}(\Oo _{\FF_e}(ph+qf)) \cong \oplus _{i=0}^{p} \Oo _{\PP^1}(q-ie),
\end{equation}
from which we may compute $h^0(\Oo _{\FF_e}(ph+qf))$. First assume $p\ge 2$. Since $D$ is irreducible, we have $q\ge pe$. This implies that 
\[
h^0(\Oo _{\FF_e}(ph+qf)) \ge h^0(\Oo_{\FF_e}(2h+2ef)) = 3e+3 \ge e+5=h^0(T\FF_e). 
\]
In fact, we get $h^0(\Oo_D(D))>h^0(T\FF_e)$ for $e\ge 2$. If $e=1$, we get the following computation
\begin{align*}
h^0(\Oo_{\FF_1}(2h+3f)) &=h^0(\Oo _{\PP^1}(3)) +h^0(\Oo _{\PP^1}(2))+h^0(\Oo _{\PP^1}(1))=9;\\
h^0(\Oo _{\FF_1}(3h+3f))&= h^0(\Oo _{\PP^1}(3)) +h^0(\Oo _{\PP^1}(2))+h^0(\Oo _{\PP^1}(1))+h^0(\Oo _{\PP^1})=10;
\end{align*}
so that we may assume that $p\le 2$; in case $p=2$, $q$ is at most two and so $D$ is in $|\Oo_{\FF_1}(2h+2f)|$. Now take $p=1$ and assume $q\ge e$. Then we have
\[
h^0(\Oo_D(D))=h^0(\Oo_{\FF_e}(h+qf))-1=2q-e+1\le e+5,
\]
and so we get $q\in \{e, e+1,e+2\}$. Now the assertion follows from Remark \ref{rat0}. 
\end{proof}


\subsection{Case of $\FF_0$}
Let us consider $\FF_0\cong \PP^1\times \PP^1$ with a polarization $\Oo _{\FF_0}(1):= \Oo _{\FF_0}(a,b)$ with $b\ge a> 0$. We have $T\FF_0 \cong \Oo _{\FF_0}(2,0)\oplus \Oo _{\FF_0}(0,2)$ and $\omega _{\FF_0}\cong  \Oo _{\FF_0}(-2,-2)$. By K\"{u}nneth formula we have $h^1(\Omega^1 _{\FF_0})=1$ and $h^1(\Omega^1 _{\FF_0}(t)) =0$ for all $t\ne 0$ and any polarization. We also have $h^1(T\FF_0) =0$. The bundle $T\FF_0$ is aCM with respect to $\Oo_{\FF_0}(1)$ if and only if $(a,b)\notin \{(1,1),(2,2)\}$.

We only consider arrangements $\Dd=\{D_1, \ldots, D_m\}$ with smooth and rational $D_i$'s; this is a necessary condition for aCM as mentioned in Remark \ref{r0}, but possibly not for $T$-aCM with respect to some polarization. So the linear systems in which each $D_i$ lives have bidegree $(p,1)$ or $(1,q)$ for some $p,q\in \ZZ_{\ge 0}$.

\begin{remark}\label{q1}
Note that $\Dd$ is aCM in degree $0$ if and only if $m\ge 2$ and the classes $\{[D_1], \ldots, [D_m]\}$ generate $H^1(\Omega_{\FF_0}^1)$; we may use the residue map in (\ref{seq1}) and $h^1(\Oo _{D_i})=0$ for all $i$. On the other hand, from the vanishing $h^1(T\FF_0)=0$ we see that $\Dd$ is $T$-aCM in degree $0$ if and only if the map 
\[
\rho: H^0(T\FF_0)\cong \CC^{\oplus 9} \rightarrow \oplus _{i=1}^{m} H^0(\Oo _{D_i}(D_i))
\]
obtained from (\ref{eqseq0}) is surjective. Note that if $D_i \in |\Oo_{\FF_0}(p,q)|$ with $p+q\in \{1,2,3\}$, then we have $h^0(\Oo_{D_i}(D_i))=2(p+q)-1$. If $\Dd$ satisfies one of the conditions above, then there exists a positive integer $k_0$ such that $\Dd$ is ($T$-)aCM for any polarization $\Oo_{\FF_0}(a,b)$ with $b\ge a\ge k_0$.
\end{remark}

\begin{remark}\label{q3}
For any polarization $\Oo_{\FF_0}(1)=\Oo_{\FF_0}(a,b)$ and an arrangement $\Dd$ with smooth and rational curve $D_i$ for each $i$, we have $h^1(\Omega _{\FF_0}^1(\log \Dd )(t)) =0$ for all $t>0$; we may use (\ref{seq1}) together with the vanishing $h^1(\Omega _{\FF_0}^1(t))=h^1(\Oo_{D_i}(t))=0$ for all $t>0$. If $\Dd$ is an arrangement of aCM type, then the vanishing $h^1(\Omega_{\FF_0}^1(-t))=0$ for $t>0$ induces an injective map $H^1(\oplus_{i=1}^m \Oo_{D_i}(-t)) \rightarrow H^2(\Omega_{\FF_0}^1(-t))$, dually a surjective map 
\[
\psi_t: H^0(T\FF_0 \otimes \Oo_{\FF_0}(at-2,bt-2)) \to H^0(\oplus_{i=1}^m \omega_{D_i}\otimes \Oo_{\FF_0}(at,bt)),
\]
which is given from the normal exact sequence associated to the embedding $D_i \hookrightarrow \FF_0$. 

For example, in the case $t=1$ with $(a,b)=(1,1)$, the surjectivity of $\psi_1$ would imply that each rational curve $D_i$ is a line in a ruling. Together with Remark \ref{q1} we get that an arrangement $\Dd$ of aCM type with respect to $\Oo_{\FF_0}(1,1)$ must consist of $p$ lines in one ruling and $q$ lines in the others with $p,q\ge 1$. Indeed, it is observed in \cite[Proposition 6.3]{BHM} that an arrangement $\Dd$ on $\FF_0$ is of aCM type with respect to $\Oo_{\FF_0}(1,1)$ if and only if $\Dd$ is of such type with $p,q\le 3$. 

On the other hand, take $(a,b)=(2,1)$. Then the surjectivity of $\psi_1$ implies that the bidegree of each curve $D_i$ is $(p,q)$ with $p,q\le 1$. Note that $h^0(\omega_{D_i}(2,1))=q(p+1)$ for $D_i \in |\Oo_{\FF_0}(p,q)|$. If no divisor $D_i$ has bidegree $(1,1)$, then we have $\Omega_{\FF_0}^1(\log \Dd) \cong \Oo_{\FF_0}(-2+k_1, 0)\oplus \Oo_{\FF_0}(0,-2+k_2)$, where $k_1$ is the number of lines with bidegree $(1,0)$ in $\Dd$ and $k_2$ is the number of lines with bidegree $(0,1)$ in $\Dd$. Then $\Dd$ is of aCM type with respect to $\Oo_{\FF_0}(2,1)$ if and only if $1\le k_1 \le 3$ and $1\le k_2 \le 2$. Now assume without loss of generality that $D_1$ has bidegree $(1,1)$. Then all the bidegrees of $D_i$ with $i\ge 2$ are same as $(1,0)$. From the surjectivity of $\psi_2$ we also get $m \le 10$. Summarizing the argument above, we get the following. 

\begin{proposition}
With respect to a fixed polarization $\Oo_{\FF_0}(1)=\Oo_{\FF_0}(2,1)$, an arrangement $\Dd=\{D_1, \ldots, D_m\}$ with $D_i\in |\Oo_{\FF_0}(a_{i1},a_{i2})|$ is of aCM type, only if one of the following holds, up to ordering. 
\begin{itemize}
\item [(i)] $(a_{i1}, a_{i2})=(1,0)$ for $1\le i \le p$ and $(a_{j1}, a_{j2})=(0,1)$ for $p+1 \le j \le m=p+q$ with $1 \le p \le 3$ and $1\le q \le 2$;
\item [(ii)] $(a_{11}, a_{12})=(1,1)$ and $(a_{i1},a_{i2})=(1,0)$ for $2 \le i \le m$ with $m\le 10$.
\end{itemize}
\end{proposition}
\end{remark}



\subsection{Case of $\FF_1$}
Note that $\FF_1$ is obtained by blowing up $\PP^2$ at a point. Set $\psi: \FF_1 \rightarrow \PP^2$ be the blow-up morphism. 

\begin{lemma}\label{k1}
For any polarization $\Oo_{\FF_1}(1)$ on $\FF_1$, we have $h^1(\Omega ^1_{\FF_1}(t)) =0$ for all $t\ne 0$.
\end{lemma} 

\begin{proof}
Set $\Oo_{\FF_1}(1):=\Oo_{\FF_1}(ah+bf)$ with $b>a>0$. For $t>0$, we have $h^1(\Oo _{\FF_1}((at-2)h + (bt-1)f)) =0$, because $at-2 \ge -1$ and  $bt-1\ge at-3$; see Remark \ref{ii0}. We also have $h^1(\Oo _{\FF_1}(ath + (bt-2)f)) =0$, because $bt-2 \ge at+a-2 \ge at-1$. Then we get $h^1(\Omega ^1_{\FF_1}(t)) =0$ from the dual of (\ref{eqy1}). Now assume $t<0$. Then by Serre's duality and Remark \ref{ii0} we have
\begin{align*}
h^1(\Omega ^1_{\FF_1}(t)) &= h^1(T\FF_1(-t)(-2h-3f))\\
&=h^1(\Oo_{\FF_1}((-at)h +(-bt-2)f)) +h^1(\Oo _{\FF_1}((-at-2)h +(-bt-1)f)) =0.
\end{align*}
from (\ref{eqy1}).
\end{proof}

\begin{lemma}\label{ii1}
For a fixed polarization $\Oo_{\FF_1}(1):= \Oo _{\FF_1}(ah+bf)$ with $b> a  \ge 2$ on $\FF_1$, the trivial arrangement is $T$-aCM with respect to $\Oo_{\FF_1}(1)$ if and only if $h^1(T\FF_1(-1))=0$
\end{lemma}

\begin{proof}
For a fixed integer $t\ge 0$, we have $h^1(\Oo _{\FF_1}(2h+f)(t)) =h^1(\Oo _{\FF_1}((at+2)h+(bt+1)f)) =0$, because $bt+1 \ge (at+2) -1$; see Remark \ref{ii0}. On the other hand, we have $h^1(\Oo _{\FF_1}(2f)(t)) =h^1(\Oo _{\FF_1}(ath+ (bt+2)f)) =0$ again by Remark \ref{ii0}. Thus (\ref{eqy1}) gives $h^1(T\FF_1(t))=0$. By the same argument as above, we may see that $h^1(T\FF_1(-t))=0$ if $at \ge 3$ and $(b-a)t\ge 2$, e.g. $a\ge 2$ and $t\ge 2$. So we get the assertion. 
\end{proof}

\begin{remark}\label{rem334}
In Lemma \ref{ii1}, if we choose the polarization $\Oo_{\FF_1}(1)=\Oo_{\FF_1}(ah+bf)$ with $a\in \{1,2\}$, then the trivial arrangement is never $T$-aCM; indeed we get $h^1(T\FF_1(-1))>0$ from (\ref{eqy1}). Note also that we have $h^1(T\FF_1(-1))=0$, if $a\ge 3$ and $b\ge a+2$. Thus we set $\Oo_{\FF_1}(1)=\Oo_{\FF_0}(ah+(a+1)f)$ with $a\ge 3$, in which case the vanishing $h^1(T\FF_1(-1))=0$ is equivalent to the vanishing $h^1(\Omega_{\FF_1}^1((a-2)h+(a-2)f))=0$. Note that $\Oo_{\FF_1}((a-2)h+(a-2)f)\cong \psi^*\Oo_{\PP^2}(a-2)$. Since $\psi$ is a birational morphism, the natural pull-back map of regular $1$-forms induces an injection $\psi^*\Omega_{\PP^2}^1(a-2) \rightarrow \Omega_{\FF_1}^1((a-2)h+(a-2)f)$. Since $\psi$ is birational and $\Omega_{\PP^2}^1(a-2)$ is locally free, we get that the map $H^0(\Omega_{\PP^2}^1(a-2)) \rightarrow H^0(\psi^*\Omega_{\PP^2}^1(a-2))$ is injective. Thus the following injective composite
\[
H^0(\Omega_{\PP^2}^1(a-2)) \hookrightarrow H^0(\Omega_{\FF_1}^1((a-2)h+(a-2)f))
\]
implies that $h^0(\Omega_{\FF_1}((a-2)h+(a-2)f))\ge h^0(\Omega_{\PP^2}^1(a-2))=(a-1)(a-3)$ by Bott's formula. On the other hand, the following long exact sequence of cohomology, obtained from the twisted dual of (\ref{eqy1}),
\begin{align*}
H^0(\Oo_{\FF_1}((a-2)h+(a-4)f)) &\to H^0(\Omega_{\FF_1}^1((a-2)h+(a-2)f))\\
&\to H^0(\Oo_{\FF_1}((a-4)h+(a-3)f)) \to H^1(\Oo_{\FF_1}((a-2)h+(a-4)f))\cong \CC
\end{align*}
gives $h^0(\Omega_{\FF_1}^1((a-2)h+(a-2)f)) = (a-1)(a-3)-\epsilon$ with $\epsilon\in \{0,1\}$. Here, we get $\epsilon=1$ if and only if $h^1(\Omega_{\FF_1}^1((a-2)h+(a-2)f))=0$, because we have $h^1(\Oo_{\FF_1}((a-4)h+(a-3)f))=0$. 
\end{remark}

From Lemma \ref{ii1} and Remark \ref{rem334} we get the following.

\begin{proposition}\label{prro}
The trivial arrangement on $\FF_1$ is $T$-aCM with respect to $\Oo_{\FF_1}(1)=\Oo_{\FF_1}(ah+bf)$ if and only if $b\ge a+2 \ge 5$.
\end{proposition}

\section{Deficiency module}
For an arrangement $\Dd$ on $X$ of dimension $n\ge 2$ with a fixed ample line bundle $\Oo_X(1)$, set
\[
H^i_*(\Dd):=\oplus_{t\in \ZZ}H^i (\Omega_X^1(\log \Dd) \otimes \Oo_X(t)).
\]
for each $i=1,\ldots, n-1$, which is a module over the ring $S=S_X:=\oplus_{t\ge 0}H^0(\Oo_X(t))$; it is called the {\it deficiency module} of degree $i$ associated to $\Dd$. Set $S_t=S_{X,t}:=H^0(\Oo_X(t))$. In this section we show that in some interesting cases these modules uniquely determine $\Dd$, which is a Torelli-type problem. Similarly we may also define {\it $T$-deficiency module} to be $H^i_*(\Dd^T):=\oplus_{t\in \ZZ}H^i(TX(-\log \Dd)\otimes \Oo_X(t))$ of degree $i$ associated to $\Dd$ to ask the same Torelli-type question. 

\begin{example}
Let $\Hh=\{H_1, \ldots, H_m\}$ be a hyperplane arrangement of Torreli type on $\PP^n$, i.e. $\Hh$ is recovered from $\Omega_{\PP^n}^1(\log \Hh)$. It was proven in \cite{valles} that this is the case when $\Hh$ does not osculate a rational normal curve with $m \ge n+3$. Indeed, each hyperplane $H_i$ is recovered as a hyperplane $H$ with $h^0(T\PP^n(-\log \Hh)_{|H})\ne 0$, called an {\it unstable hyperplane}. Let $f_H \in \CC[x_0,\dots ,x_n]$ be the equation of a hyperplane $H$. Since $h^0(T\PP^n(-\log \Hh) )=0$, we have $h^0(T\PP^n(-\log \Hh)_{|H})\ne 0$ if and only if the induced map 
\[
f_{H*}: H^1(T\PP^n(-\log \Hh)(-1)) \rightarrow H^1(T\PP^n(-\log \Hh))
\]
by the multiplication by $f_H$ is not injective. Thus the set of all unstable hyperplanes of $\Omega_{\PP^n}^1(\log \Hh)$ can be described by a small part of the $T$-deficiency module of $\Dd$ and also by the deficiency module of $\Dd$ due to Serre's duality. 
\end{example}


\begin{example}
Fix an arrangement $\Dd ' = \{D'_1,\dots ,D'_m\}$ on $\PP^n$ with $n\ge 2$, whose deficiency module
determines $\Dd'$. Let $\pi : X\rightarrow \PP^n$ be the blow-up at finitely many points $p_1,\dots ,p_s$ with $E_i:=\pi^{-1}(p_i)$ such that none of them is contained in one component of $\Dd'$. Letting $D_i$ be the strict transformation of $D_i'$, we set $\Dd=\{D_1, \ldots, D_m\}$ an arrangement on $X$. To define the deficiency modules of $\Dd$, we need to fix an ample line bundle $\Oo _X(1)$ on $X$. We have $\Oo _X(1) \cong \pi ^\ast \Oo _{\PP^n}(e_0)(-e_1E_1-\cdots -e_sE_s)$ with $e_1 \ge \ldots \ge e_s>0$. Note that not every choice of $(e_0, \ldots, e_s)$ gives an ample line bundle, e.g. we need $e_0>e_i$ for each $i\ge 1$. Now assume $s\ge 2$. Consider the line $L$ containing $\{p_1, p_2\}$ and its strict transform $\widetilde{L}$. Then we get $\deg (\Oo_{\widetilde{L}}(1))\le e_0-e_1-e_2$ with equality if and only if $p_i \not\in L$ for all $i>2$. Since $\deg (\Oo_{\widetilde{L}}(1))$ is also positive, we get $e_0>e_1+e_2$. For the same reason, if $s\ge 3$ and $\{p_i, p_j,p_h\}$ are collinear with $|\{i,j,k\}|=3$, then we get $e_0>e_i+e_j+e_k$. In case $n=2$ and $s\ge 5$ we get $2e_0 > e_1+e_2+e_3+e_4+e_5$, because any five points of the plane are contained in a conic. 
\end{example}

\begin{proposition}
Let $X$ be an abelian variety and choose an arrangement $\Dd =\{D_1,\ldots ,D_m\}$ on $X$ such that the classes $[D_1],\ldots ,[D_m]$ are linearly independent in $H^2(X,\CC)$. Then $\Dd$ is uniquely determined by the isomorphism class of $\Omega ^1_X(\log \Dd )$.
\end{proposition}

\begin{proof}
By assumption the coboundary map $\oplus _{i=1}^{m}H^0( \Oo _{D_i})\rightarrow H^1(\Omega ^1_X)$ induced by (\ref{seq1}) is injective, which implies that the natural map $H^0(\Omega ^1_X) \rightarrow H^0(\Omega ^1_X(\log \Dd ))$ is an isomorphism. In particular, the sheaf $\oplus _{i=1}^{m} \Oo _{D_i}$
is isomorphic to the cokernel of the evaluation map 
\[
H^0(\Omega ^1_X(\log \Dd ))\otimes \Oo _X\cong \Omega_X^1 \to \Omega ^1_X(\log \Dd ),
\]
concluding the assertion. 
\end{proof}

As mentioned in Remark \ref{arar} the notion of aCM and reconstructability for an arrangement $\Dd$ obviously depend on the choice of a polarization on $X$. No arrangement may be reconstructable for all polarizations on $X$, as shown by the following well-known observation. At the opposite side of reconstructible arrangements there are the $1$-Buchsbaum and $1^T$-Buchsbaum arrangements in the sense of the following definition.

\begin{definition}
For a fixed ample line bundle $\Oo _X(1)$, an arrangement $\Dd$ is said to be {\it $1$-Buchsbaum} (resp. {\it  $1^T$-Buchsbaum}) in degree $i$ with respect to $\Oo _X(1)$ if the ring $S_X$ acts trivially on each $H^i_*(\Dd)$ (resp. $H^i_*(\Dd^T)$). If $\Dd$ is $1$-Buchsbaum (resp. $1^T$-Buchsbaum) in every degree $i=1,\ldots, n-1$, then we say that it is $1$-Buchsbaum (resp. $1^T$-Buchsbaum). 
\end{definition}

\begin{remark}
If the polarization is subcanonical, then the two notions of $1$-Buchsbaum and $1^T$-Buchsbaum coincide. Remark \ref{arar} shows that every arrangement is $1$-Buchbaum and $1^T$- Buchsbaum for some polarization. As Example \ref{exx1} shows, both notions are clearly weaker than the notion of aCM.
\end{remark}

\begin{example}\label{exx1}
Let $X$ be an abelian variety of dimension $n$ with a fixed ample line bundle $\Oo_X(1)$. Since $h^i(\Oo _X(t)) =0$ for all $t\in \ZZ\setminus \{0\}$ and $i=1,\ldots, n-1$, the trivial arrangement $\Dd =\emptyset$ is $1$-Buchsbaum, but not aCM.
\end{example}




\begin{proposition}\label{non1}
Let $X$ be a Del Pezzo surface of degree $N$ with $N \in \{5,6,7,8\}$, as the blow-up $\pi : X \rightarrow \PP^2$ at $(9-N)$-points $p_1,\ldots, p_{9-N}$. Setting $D_i:=\pi^{-1}(p_i)$ for each $i$, consider an arrangement $\Dd=\{D_1, \ldots, D_{9-N}\}$. Then any subarrangement $\Dd'\subset \Dd$ is aCM in degree zero with respect to $\Oo _X(1):= \omega _X^\vee$. In particular, $\Dd'$ is $1$-Buchsbaum. 
\end{proposition}

\begin{proof}
Note that $X$ is obtained by blowing up $\pi : X \rightarrow \PP^2$ at $(9-N)$-points $p_1,\ldots, p_{9-N}$ such that no three of them are collinear. This implies $h^1(TX) =0$. We also have the cotangent exact sequence
\[
0\to \pi^*\Omega_{\PP^2}^1 \to \Omega_X^1 \to \oplus_{i=1}^{9-N} \epsilon_{i*}\Omega_{D_i}^1\to 0
\]
where $\epsilon_i: D_ i \rightarrow X$ is the embedding, from which we get $h^1(\Omega_X^1(t))=0$ for all $t>0$ and $h^1(\Omega_X^1)=9-N$. This implies by $TX \cong \Omega_X^1(1)$ and Serre's duality that $h^1(\Omega_X^1(t))=0$ for all $t\ne 0$. Again by Serre's duality we have $h^1(TX(t))=0$ for all $t\ne -1$. 

Now without loss of generality we may set $\Dd'=\{D_1, \ldots, D_m\}$ with $m\le 9-N$. Since each $D_i$ is a smooth rational curve, we have $h^1(\Oo _{D_i}(t)) =0$ and (\ref{seq1}) implies $h^1(\Omega ^1_X(\log \Dd' )(t)) =0$ for $t>0$. Now assume $t<0$ and set $t' =-t$. By Serre's duality we need to prove that $h^1(TX(-\log \Dd' )(t'-1)) =0$. First consider the case $t'=1$. Since $D_i$ is a smooth and rational  curve with $\Oo_{D_i}(D_i) \cong \Oo_{D_i}(-1)$, we have $h^1(TX(-\log \Dd' ))=0$ from (\ref{eqseq0}). On the other hand, we have $h^2(TX(-\log \Dd' )(-1)) = h^0(\Omega _X^1(\log \Dd'))$. Since $X$ is smooth and rational, we have $h^0(\Omega _X^1) =0$. Note that each $D_i$ is a different exceptional divisor, we get that $[D_1], \ldots ,[D_m]$ are linearly independent in $H^2(X,\CC)$ and so in $H^1(\Omega _X^1)$. Thus the exact sequence (\ref{seq1}) gives $h^0(\Omega _X^1(\log \Dd')) =0$. This implies that the bundle $TX(-\log \Dd')$ is $1$-regular, and by the Castelnuovo-Mumford regularity lemma we get that $h^1(TX(-\log \Dd')(t'-1))=0$ for all $t'>0$. 
\end{proof}

Let $\ell : V\rightarrow W$ be a linear map between finite-dimensional vector space. We say that $\ell$ has \emph{maximal rank} if it is either injective or surjective. In this case we have
\[
\dim \mathrm{ker}(\ell )= \max \{0, \dim V -\dim W\} \text{ and }\dim \mathrm {Im}(\ell )= \min \{\dim V,  \dim W\}.
\]
\noindent In general, for a standard graded algebra $S$, i.e. it is generated by $S_1$, a finite-dimensional graded $S$-module $ M= \oplus _{t\in \ZZ} M_t$ is said to have the \emph{weak Lefschetz property} (resp. \emph{strong Lefschetz property}) if for a general $f\in S_1$ the linear maps $M_t \to M_{t+1}$ induced by $f$ have maximal rank (resp. for every integer $q>0$ the linear maps $M_t \to M_{t+q}$, induced by $f^q$ have maximal rank) for all $t\in \ZZ$. 

\begin{definition}\label{ddee}
An arrangement $\Dd$ on $X$ is said to have {\it very strong Lefschetz property} in degree $i$ if for any $q>0$ and a general element $f\in S_q$ the linear map
\begin{equation}\label{defmap}
\mu_{\Dd}(f, p, i): H^i(\Omega_X^1(\log \Dd)(p)) \to H^i(\Omega_X^1(\log \Dd)(p+q))
\end{equation}
induced by $f$ has maximal rank for every $p\in \ZZ$. 
\end{definition}

\begin{remark}
For a positive integer $q$ and $z\in S_1^{\times}=H^0(\Oo_X(1))^{\times }$, the multiplication by $z^q$ induces a linear map $\mu_{\Dd}( z^q,p,i)$ for each $i\in \{0,\ldots, n\}$ and $p\in \ZZ$. Since the scalar multiplication on $z$ produces no change in the rank of $\mu_{\Dd}(z^q, p, i)$, we may consider a natural stratification
\[
\Zz_0(\Dd, q, p, i) \subset \Zz_1(\Dd, q, p, i) \subset \cdots \subset \PP H^0(\Oo_X(1))
\]
where $\Zz_j(\Dd, q, p, i)$ is the set of $z\in H^0(\Oo_X(1))^{\times}$, up to scalar, whose corresponding map $\mu_{\Dd}(z^q, p, i)$ has rank less than $j+1$. For instance, $\Dd$ is strictly $k$-Buchsbaum if and only if we have $\Zz_0(\Dd, q+1, p, i)=\PP H^0(\Oo_X(1))$ for all $p\in \ZZ$ and $i\in \{1,\ldots, n-1\}$, but there is at least one $i\in \{1,\dots ,n-1\}$ and $p\in \ZZ$ with $\Zz_0(\Dd, q, p, i)\subsetneq \PP H^0(\Oo_X(1))$ for some $p\in \ZZ$ and $i\in \{1,\ldots, n-1\}$. Note that for each $i>0$ and $p\in \ZZ$ there is an integer $q$ such that $\Zz_0(\Dd, q, p, i)=\PP H^0(\Oo_X(1))$, because $\Oo_X(1)$ is ample. Now for a fixed $z\in H^0(\Oo_X(1))^{\times}$, the {\it order} of the deficiency modules of $\Dd$ with respect to $z$, denoted by $\mathrm{ord}_{\Dd,z}$, is the minimal integer $q$ such that $\Zz_0(\Dd, q, p, i)=\PP H^0(\Oo_X(1))$ for all $i=1,\dots ,n-1$ and all $p,q\in \ZZ$, with the convention that $\mathrm{ord}_{\Dd, z}=0$ for $\Dd$ aCM. 
\end{remark}

\begin{remark}
In Definition \ref{ddee}, the graded algebra $S=S_X$ is not necessarily standard, i.e. the natural map $H^0(\Oo_X(1))^{\otimes k} \rightarrow H^0(\Oo_X(k))$ may not be surjective for some $k \ge 2$. It is clear that $\Dd$ has very strong Lefschetz property in degree $i$ if there exists at most one integer $t$ such that $H^i(\Omega_X^1(\log \Dd)(t))\ne 0$. These are equivalent conditions if $\Dd$ is $1$-Buchsbaum in degree $i$. Note also that the analogues of all these notions may be defined by considering $TX(-\log \Dd)$ instead of $\Omega ^1_X(\log \Dd )$.
\end{remark}

\begin{remark}\label{u1}
Let $D$ be an integral curve on a smooth projective variety $X$. Fix any nonzero element $f\in S_q$ with $q>0$. Then its associated map $f_{*,t}: H^1(\Oo_D(t)) \rightarrow H^1(\Oo_D(t+q)$ is surjective for any $t\in \ZZ$, because its dual map $H^0(\omega_D(-t-q)) \rightarrow H^0(\omega_D(-t))$ is injective; here we use that $D$ is an integral curve. 
\end{remark}

\begin{example}\label{u1.1}
Let $X$ be a smooth  K3 surface with a fixed ample line bundle $\Oo_X(1)$. We fix a positive integer $a$ with $h^1(\Omega_X^1(a))=0$ and consider an arrangement $\Dd=\{D_1, \ldots, D_m\}$ with each $D_i\in |\Oo_X(a)|$; if $\Oo_X(a)$ is very ample, then we may find such an arrangement for any $m$. By the adjunction formula we have $\omega _{D_i} \cong \Oo _{D_i}(a)$ and so $2p_a(D_i) -2 = \deg (\Oo _{D_i}(a)) = a^2\deg (X)$ for each $i$. In particular, we get $h^1(\Oo _{D_i}) = 1 + a^2\deg (X)/2\ge 2$. By Hodge theory we have $h^2(\Omega ^1_X) = h^1(\omega _X)=0$. Thus (\ref{seq1}) gives 
\[
h^1(\Omega _X^1(\log \Dd ))=h^1(\Omega_X^1)+ mp_a(D_i)-\rho
\]
for any $i$, where $\rho$ is the dimension of the linear span of $\{[D_1], \ldots, [D_m]\}$ in $H^1(\Omega_X^1)$. By Serre's duality we have $h^2(\Omega ^1_X(a)) =h^0(TX(-a)) =0$. Since $\omega _{D_i} \cong \Oo _{D_i}(a)$ for each $i$, we have $h^1(\Oo _{D_i}(a)) =1$ and (\ref{seq1}) gives $h^1(\Omega ^1_X(\log \Dd )(a)) =m$. Let $f_i\in H^0(\Oo _X(a))\setminus \{0\}$ be an element defining $D_i$; by assumption $f_i$ and $f_j$ are not proportional for $i\ne j$. For any nonzero element $f\in H^0(\Oo _X(a))$, we have a map 
\[
f_* : H^1(\Omega_X^1(\log \Dd)) \to H^1(\Omega_X^1(\log \Dd)(a))\cong H^1(\oplus_{i=1}^m \Oo_{D_i}(a))
\]
that factors through $H^1(\oplus_{i=1}^m\Oo_{D_i})$, which also fits into the following commutative diagram
\[
\begin{array}{ccccccc}
&0 & &0& & & \\
& \downarrow& &\downarrow & & &\\
0\to &  \Omega_X^1     &\to  &\Omega_X^1(\log \Dd)  &  \to  &\oplus_{i=1}^m \Oo_{D_i} &\to 0 \\
& \downarrow& &\downarrow & &\downarrow & \\
0\to &  \Omega_X^1(a)       &\to  &\Omega_X^1(\log \Dd)(a)  &  \to  &\oplus_{i=1}^m \Oo_{D_i}(a)  &\to 0 \\
&\downarrow& &\downarrow & &\downarrow &\\
0\to &  \Omega_X^1(t+q)_{|D}       &\to  &\Omega_X^1(\log \Dd)(t+q)_{|D}  &  \to  &\oplus_{i=1}^m \Oo_{D \cap D_i}(t+q)  &\to 0 \\
&\downarrow& &\downarrow & & \downarrow& \\
&0 &          &0& & 0,& \\
\end{array}
\]
where the first two rows are the exact sequence (\ref{seq1}) twisted by $0$ and $a$, respectively, and all vertical maps are induced by the multiplication by $f$ with $D$ as its associated divisor. Note that the right vertical sequence is not necessarily short exact. Recall that $h^2(\Omega ^1_X) =h^2(\Omega _X^1(a)) =h^1(\Omega^1_X(a)) =0$. If $f$ is not a scalar multiple of $f_i$ for some $i$, then the map $f_*$ is surjective by Remark \ref{u1}. In case when $f$ is a scalar multiple of one of $f_i$'s, the map $f_*$ has corank one. Thus the deficiency module $H^1_*(\Dd)$ determines $\Dd$. 
\end{example}

In the next example for Enriques surfaces we see how a certain arrangement is determined by the deficiency module of $\Dd$, or to be precise, by the deficiency module of a twist of $\Omega ^1_X(\log \Dd )$ by $\omega_X$ a line bundle of order two. In particular, the arrangement $\Dd$ in consideration is uniquely determined by the isomorphism class of $\Omega ^1_X(\log \Dd )\otimes \omega _X$ and so by the isomorphism class of $\Omega ^1_X(\log \Dd )$.

\begin{proposition}
Let $X$ be an Enriques surface with a fixed ample line bundle $\Oo _X(1)$. Fix an arrangement $\Dd=\{ D_1, \ldots, D_m\}$ with each $D_i \in |\Oo_X(a)\otimes \omega_X|$ for some positive integer $a$ with $h^1(\Omega_X^1(a)\otimes \omega_X)=0$. Then the multiplication map 
\[
\gamma : H^0(\omega _X(a))\otimes H^1(\Omega ^1_X(\log \Dd )\otimes \omega _X) \to H^1(\Omega ^1_X(\log \Dd )(a))
\]
determines $\Dd$.
\end{proposition}

\begin{proof}
Note that we use that $\omega_X^{\otimes 2} \cong \Oo_X$ in the definition of $\gamma$. Let $f_i\in H^0(\Oo _X(a)\otimes \omega_X)$ be a nonzero equation defining $D_i$. By assumption $f_i$ and $f_j$ are not proportional for $i\ne j$. Now for each $f\in H^0(\Oo _X(a)\otimes \omega _X)$, consider a map 
\[
\gamma _f: H^1(\Omega ^1_X(\log\Dd )\otimes \omega_X) \rightarrow H^1(\Omega ^1_X(\log \Dd )(a))
\]
defined by $\alpha \mapsto \gamma (z,\alpha)$. By Serre's duality we have $h^2(\Omega ^1_X\otimes \omega _X) = h^0(TX) =0$. Thus from (\ref{seq1}) we get
\[
h^1(\Omega_X^1(\log \Dd )\otimes \omega _X)=h^1(\Omega_X^1\otimes \omega_X)+mh^1(\Oo_{D_i}\otimes \omega _X)
\]
for any $i$, because we have $h^0(\omega_X\otimes \Oo_{D_i})=0$. By Serre's duality we also have $h^2(\Omega_X^1(a))=h^0(TX(-a)\otimes \omega_X)=0$. Since $\omega_{D_i} \cong \Oo_{D_i}(a)$ for each $i$ by the adjunction formula, we get $H^1(\Oo_{D_i}(a)) \cong H^0(\Oo_{D_i})^\vee$ and so we get that $H^1(\Omega_X^1(\log \Dd)(a)) \cong H^1(\oplus_{i=1}^m \Oo_{D_i}(a))$ is $m$-dimensional. Note that the map $\gamma_f$ factors through $H^1(\oplus_{i=1}^m \Oo_{D_i}\otimes \omega_X)$. Thus as in Example \ref{u1.1} we get that $\gamma_f$ for $f\ne 0$ is surjective if and only if $f$ is not a scalar multiple of $f_i$ for some $i$; the map $\gamma_{f_i}$ has corank one for each $i$.
\end{proof}

\bibliographystyle{amsplain}

\begin{thebibliography}{9}

\bibitem{BHM}
E.~Ballico, S.~Huh and F.~Malaspina, \emph{A Torelli-type problem for logarithmic bundles over projective varieties}, Q. J. Math. \textbf{66} (2015), 417--436.

\bibitem{Bom}
E.~Bombieri, \emph{Canonical models of surfaces of general type}, Publ. Math. IHES \textbf{42} (1973), 171-220.

\bibitem{brion}
M.~Brion, \emph{On automorphisms and endomorphisms of projective varieties}, Automorphisms in Birational and Affine Geometry, Springer Proc. Math. Stat. \textbf{79} (2014), 59--82. 

\bibitem{De}
P.~Deligne, \emph{Th\'eorie de Hodge II}, Inst. Hautes \'Etudes Sci. Publ. Math. \textbf{40} (1971), 5--58.

\bibitem{d} 
M.~Demazure, \emph{Surfaces de Del Pezzo, I, II, III, IV, V, S\'{e}minaire sur les Singularit\'{e}s
des Surfaces, Palaiseau, France 1976--1977}, Lect. Notes in Mathematics 777, Springer, Berlin, 1980.

\bibitem{D}
I.~Dolgachev, \emph{Logarithmic sheaves attached to arrangements of hyperplanes}, J. Math. Kyoto Univ. \textbf{47} (2007), no.~1, 35--64.

\bibitem{Dol}
I.~Dolgachev, \emph{Classical algebraic geometry: a modern view}, Cambridge University Press, Cambridge (2012).

\bibitem{DK}
I.~Dolgachev and M.~Kapranov, \emph{Arrangements of hyperplanes and vector bundles on {$\bold P\sp n$}}, Duke Math. J. \textbf{71} (1993), no.~3, 633--664.

\bibitem{GH} 
J.~P. Griffiths and J.~Harris, \emph{Principles of algebraic geometry}, Pure and Applied Mathematics. Wiley-Interscience [John Wiley \& Sons], New York, 1978. xii+813 pp. ISBN: 0-471-32792-1 

\bibitem{hart} 
R.~Hartshorne, \emph{Algebraic Geometry}, Springer-Verlag, Berlin--Heidelberg--New York, 1977.

\bibitem{Horrocks}
G.~Horrocks, \emph{Vector bundles on the punctual spectrum of a local ring}, Proc. Lond. Math. Soc. (3) \textbf{14} (1964), 689--713.

\bibitem{mizu} 
M.~Mizukami, \emph{Birational mappings from quartic surfaces to Kummer surfaces}, Master Thesis at University of Tokyo, 1975 (in Japanese).

\bibitem{OT}
P.~Orlik and H.~Terao, \emph{Arrangements of hyperplanes}, Grundlehren der Mathematischen Wissenschaften, 300. Springer-Verlag, Berlin, 1992. xviii+325 pp. 

\bibitem{ssvl} 
M.~Sch\"{u}tt, T.~Shioda and R.~van Luijk, \emph{Lines on Fermat surfaces}, J. Number Theory \textbf{130} (2010), 1939--1963.

\bibitem{s}
T.~Shioda, \emph{On the Picard number of a Fermat surface}, J. Fac. Sci. Univ. Tokyo Sect. IA Math. \textbf{28} (3) (1982), 725--734.

\bibitem{Te}
H.~Terao, \emph{Arrangements of hyperplanes and their freeness I,II}, J. Fac. Sci. Univ. Tokyo Sect. IA Math. \textbf{27} (1980), no.~2, 293--320. 

\bibitem{valles}
J.~Vall\'es, \emph{Nombre maximal d'hyperplans instables pour un fibr\'e de Steiner}, Math. Zeit. \textbf{233} (2000), 507--514. 

\bibitem{voi} 
C.~Voisin, \emph{Hodge Theory and Complex Algebraic Geometry II}, Cambridge University Press, Cambridge UK, 2007.

\end{thebibliography}
\providecommand{\bysame}{\leavevmode\hbox to3em{\hrulefill}\thinspace}
\providecommand{\MR}{\relax\ifhmode\unskip\space\fi MR }
\providecommand{\MRhref}[2]{%
  \href{http://www.ams.org/mathscinet-getitem?mr=#1}{#2}
}
\providecommand{\href}[2]{#2}

\end{document}